\documentclass[]{interact}

\usepackage{amsmath,amssymb,amsthm,mathtools,graphicx,microtype}
\usepackage{cite}

\theoremstyle{plain}
\newtheorem{theorem}{Theorem}[section]
\newtheorem{lemma}[theorem]{Lemma}

\theoremstyle{definition}
\newtheorem{definition}[theorem]{Definition}

\theoremstyle{remark}
\newtheorem{remark}[theorem]{Remark}

\newtheorem{assumption}[theorem]{Assumption}

\usepackage{bm,pgfplots}
\pgfplotsset{compat=1.17}
\usepackage{array,hyperref,enumitem}
\usepackage{orcidlink}

\newcolumntype{M}[1]{>{\centering\arraybackslash}m{#1}}

\newcommand{\dx}{\,\mathrm{d}x}
\newcommand{\dt}{\,\mathrm{d}t}
\newcommand{\ddt}{\frac{\mathrm{d}}{\mathrm{d}t}}
\newcommand{\ds}{\,\mathrm{d}s}
\newcommand{\R}{\mathbb{R}}
\newcommand{\Div}{\operatorname{div}}

\definecolor{myblue}{RGB}{31,119,180}
\definecolor{myorange}{RGB}{255,127,14}
\definecolor{mygreen}{RGB}{44,160,44}

\renewcommand{\rho}{\varrho}

\definecolor{color0}{HTML}{4E79A7}
\definecolor{color1}{HTML}{F28E2B}
\definecolor{color2}{HTML}{E15759}
\definecolor{color3}{HTML}{76B7B2}
\definecolor{color4}{HTML}{59A14F}
\definecolor{color5}{HTML}{EDC948}
\definecolor{color6}{HTML}{B07AA1}
\definecolor{color7}{HTML}{FF9DA7}
\definecolor{color8}{HTML}{9C755F}
\definecolor{color9}{HTML}{BAB0AC}

\begin{document}

\articletype{\vspace{-1.5cm}}

\title{Global weak solutions of a one-sided degenerate Cahn--Hilliard model for traction-driven digit morphogenesis}

\author{
\name{Marvin Fritz\orcidlink{0000-0002-8360-7371}}
\email{marvin.fritz@univie.ac.at}
\affil{Faculty of Mathematics, University of Vienna, Vienna, Austria}
}

\maketitle

\begin{abstract}
We study a one-sided degenerate Cahn--Hilliard equation with anisotropic traction flux, arising as a reduced continuum description of mechanically biased cell interactions in digit-forming organoids. The equation combines a one-sided degenerate mobility with a density-weighted anisotropic higher-order transport term. This traction term is not generated by the variational derivative of the Cahn--Hilliard energy and therefore produces sign-indefinite contributions in the energy balance. For nonnegative initial data, we prove the global-in-time existence of nonnegative weak solutions. The proof combines an energy estimate for the diffusive flux with a mobility-matched entropy method adapted to the vacuum degeneracy. A key point is that the entropy variable cancels the mobility, turning the anisotropic traction contribution into a coercive first-order term in the entropy inequality, while the energy estimate supplies a weighted control of the diffusive flux.
\end{abstract}

\section{Introduction}

We study the degenerate Cahn--Hilliard-type equation
\begin{align*}
\partial_t\rho
&=\Div(\rho\nabla q)+\varepsilon^{-1}\eta^2 \Div(\rho D(\nu)\nabla\rho),\\
q
&=\varepsilon^{-1}W'(\rho)-\varepsilon\Delta\rho,
\end{align*}
on bounded domains with no-flux boundary conditions, where \(D(\nu)\) is a prescribed bounded symmetric nonnegative tensor field. The model was derived in \cite{Tsutsumi2025} as a reduced continuum description of mechanically biased cell interactions in digit-forming organoids. Further, a structure-preserving scheme for the regularized model was analyzed in \cite{fritz2026structure}. In the present paper, our main goal is the existence analysis of this fourth-order equation with degenerate mobility and anisotropic transport. For general overviews of the Cahn--Hilliard equation and its variants, we refer to the monograph \cite{miranville2019cahn} and the surveys \cite{wu2022review,fritz2023tumor,NovickCohen2008}.

The present problem differs from the standard settings in two essential ways. First, the mobility \(m(\rho)=\rho\) degenerates only at the vacuum state, so the natural entropy structure is one-sided and tied to nonnegativity. Second, the additional traction term is not generated by the variational derivative of the Cahn--Hilliard energy. Thus the equation is not a gradient flow of \eqref{eq:E} with respect to the mobility-induced metric; rather, the traction flux acts as a prescribed conservative transport perturbation of the degenerate Cahn--Hilliard dynamics. The main analytical task is therefore to recover compactness and identify the diffusive flux despite the degeneracy of the mobility and the sign-indefinite traction contribution in the energy balance.

Our main result proves the global-in-time existence of nonnegative weak solutions for nonnegative initial data on bounded domains with no-flux boundary conditions. The proof combines energy estimates with a mobility-matched entropy method adapted to the one-sided degeneracy and to the anisotropic traction contribution. More precisely, the entropy inequality controls the fourth-order term and provides the compactness needed to pass to the limit, while the energy estimate yields a weighted bound on the diffusive flux. To the best of our knowledge, this is the first global existence result for this reduced organoid equation and, more generally, for this class of degenerate Cahn--Hilliard equations with prescribed anisotropic traction tensors.

The equation is also related to perturbed Cahn--Hilliard systems in which phase separation interacts with transport, anisotropy, or forcing. Diffuse-interface descriptions of fingering and interfacial instabilities arise, for example, in Hele--Shaw and porous-media settings \cite{cueto_juanes_2014,wise_2010_chhs,wang_zhang_2013,AndersonMcFaddenWheeler1998,LowengrubTruskinovsky1998}, while advective Cahn--Hilliard models illustrate how transport can compete with coarsening and enhance filamentary structures \cite{naraigh_thiffeault_2007,feng_feng_iyer_thiffeault_2019,fritz2019unsteady}. In the present case, however, the perturbation is a density-weighted anisotropic higher-order flux, and this specific structure is what drives both the analysis and the observed protrusive dynamics.

The biological motivation comes from recent work on digit-forming organoids and related mechanochemical models of tissue self-organization. Digit patterning has long been connected to reaction--diffusion and mechanochemical mechanisms, beginning with Turing's seminal ideas \cite{Turing1952} and subsequent developments in developmental biology and morphogenesis \cite{KondoMiura2010,Sheth2012,Raspopovic2014,OsterMurrayHarris1983,MurraySwanson2003,MurrayOster1984}. More recently, organoid systems have provided experimentally accessible settings in which growth, signalling, and mechanics can be perturbed and quantified in vitro; see, for example, \cite{LancasterKnoblich2014,Clevers2016,sato2009single,Okuda2018,PerezGonzalez2021,Yang2021,Gjorevski2022,Nikolaev2020}. Motivated by the experiments in \cite{Tsutsumi2025}, an agent-based model was proposed in which mechanically biased cell--cell traction generates protrusive digit-like morphologies. The reduced one-field equation studied here corresponds to a regime in which the underlying cell populations are effectively segregated and the macroscopic orientation field \(\nu\) is treated as prescribed.

The rest of the paper is organized as follows. In Section~\ref{sec:model} we introduce the reduced continuum model and state the standing assumptions. Section~\ref{sec:analysis} contains the definition of weak solutions and the proof of global existence based on regularization, energy estimates, entropy inequalities, and compactness arguments.
We conclude in Section~\ref{sec:conclusion} with a discussion of open problems and possible extensions.

\section{Mathematical modeling}\label{sec:model}

Let \(\Omega\subset\mathbb{R}^d\) (\(d\in\{2,3\}\)) be a bounded domain with outer unit normal \(n\).
We study a reduced continuum model for a distal cell-density variable
\(\rho=\rho(x,t)\ge 0\) for \((x,t)\in \Omega\times(0,T)\),
where \(\rho\) is interpreted as a local number density of motile cells.
The dynamics is driven by a diffuse-interface free energy of Ginzburg--Landau type:
\begin{equation}\label{eq:E}
E(\rho):=\int_\Omega\Big(\varepsilon^{-1}W(\rho)+\frac{\varepsilon}{2}|\nabla\rho|^2\Big)\dx,
\end{equation}
where \(\varepsilon>0\) is the diffuse-interface length scale and \(W\) is a
potential favoring two phases such as low and high density.
The associated chemical potential $q$ is given as the variational derivative of \(E\).

We introduce mechanically biased cell--cell interactions through a prescribed orientation field
\(\nu=\nu(x,t)\in\mathbb{R}^d\).
For instance, \(\nu\) may be a normalized macroscopic direction extracted from the underlying agent-based model.
From \(\nu\) we build a symmetric positive semidefinite tensor \(D(\nu)\) describing anisotropic traction.
A prototypical example is \(D(\nu)=\nu\otimes\nu\), but our analysis allows general bounded
tensors \(D(\nu)=D(\nu)^\top\ge 0\). With mobility \(m(\rho)=\rho\) and traction strength \(\eta\ge 0\), the reduced distal model of \cite{Tsutsumi2025} then reads
\begin{align}
\partial_t\rho
&=\Div\bigl(\rho \nabla q\bigr)
+\varepsilon^{-1}\eta^2 \Div\bigl(\rho D(\nu) \nabla\rho\bigr),
\label{eq:PDE}
\\
q
&=\varepsilon^{-1}W'(\rho)-\varepsilon\Delta\rho.
\label{eq:q2}
\end{align}
Equivalently, with the total flux
\[
J(\rho,q):=\rho\nabla q+\varepsilon^{-1}\eta^2 \rho D(\nu)\nabla\rho,
\]
the equation reads \(\partial_t\rho=\Div J\). We prescribe the initial condition \(\rho(\cdot,0)=\rho_0\) in \(\Omega\).
We impose homogeneous Neumann/no-flux conditions
\begin{equation}\label{eq:BC}
\nabla\rho\cdot n = 0,
\qquad
J(\rho,q)\cdot n=0
\qquad\text{on }\partial\Omega\times(0,T).
\end{equation}
These boundary conditions ensure that all integrations by parts below produce no boundary contributions. Furthermore,
under \eqref{eq:BC}, the total mass \(\int_\Omega \rho(x,t)\dx\) is conserved,
which corresponds to a fixed total number of motile cells in the reduced model. 

Additionally, when \(\eta=0\),
\eqref{eq:PDE}--\eqref{eq:q2} reduces to a degenerate-mobility
Cahn--Hilliard equation as in \cite{ElliottGarcke1996}. In this case the
equation is the mobility-weighted gradient flow of the energy \eqref{eq:E}
and satisfies the standard energy-dissipation property. For \(\eta>0\), the
traction-driven term acts as an active conservative transport contribution.
It is not generated by the variational derivative of \eqref{eq:E}, and
therefore the energy identity contains an additional sign-indefinite term.
Similar perturbed energy-dissipation situations arise when Cahn--Hilliard
dynamics is subject to external forcing or feedback control, where one
typically retains an energy inequality with additional right-hand-side terms
that must be estimated. We refer to
\cite{azmi2025stabilization,egger2025feedback} for related stability
mechanisms and estimates in stabilized Cahn--Hilliard settings.

\begin{remark}[Derivation of the model]
The reduced equation \eqref{eq:PDE}--\eqref{eq:q2} is motivated by the agent-based model proposed in \cite{Tsutsumi2025}
for digit-forming organoids. In that model, distal and proximal cells interact through short-range repulsion and
cell--cell tractions whose strengths depend on cell type; in addition, tractions are biased relative to
morphogen gradients (for example, along a chemoattractive cue and/or orthogonal to a cue to model convergent extension).
A mean-field limit in which the number of cells tends to infinity while the interaction radius scales like a
small length \(\varepsilon\) yields PDEs for the distal and proximal densities, say \(\rho_\mathrm{D}\) and \(\rho_\mathrm{P}\),
whose supports encode the evolving tissue shapes \cite{Tsutsumi2025}.

In the asymptotic expansion described in \cite{Tsutsumi2025}, the leading-order terms (of order \(\varepsilon^{-1}\))
have an advective character: \(\rho_\mathrm{D}\) drifts up the prescribed gradient direction while \(\rho_\mathrm{P}\) drifts down it.
These terms vanish in the fully segregated regime \(\rho_\mathrm{D}\rho_\mathrm{P}=0\), and the next-order dynamics become a
modification of the Cahn--Hilliard equation in which traction bias produces an additional anisotropic transport
term capable of triggering fingering-type protrusions.
In the present work we focus on the corresponding reduced one-field description, obtained by restricting to a
single (distal) population and treating the macroscopic orientation field \(\nu\) as prescribed.
\end{remark}

\begin{remark}[Interpretation] The Cahn--Hilliard part \(\partial_t\rho=\Div(\rho\nabla q)\) can be viewed as density-driven rearrangement in a
phase-separating medium: the double-well structure of \(W\) favors segregation into low- and high-density regions,
while the gradient term penalizes sharp interfaces. The choice of mobility \(m(\rho)=\rho\) reflects that transport
is mediated by the presence of cells and degenerates in cell-free regions, which is consistent with the tissue
occupancy interpretation and is crucial for nonnegativity preservation in the weak formulation.

The second flux \(\varepsilon^{-1}\eta^2 \rho D(\nu)\nabla\rho\) represents mechanically biased traction-driven
transport. Here \(\eta\ge 0\) quantifies the strength of the bias, and \(D(\nu)\) encodes the anisotropy induced by a preferred direction field \(\nu\)
(for example, a morphogen-gradient direction extracted from the coarse-grained setting).
In organoid-like configurations one may take \(\nu\) to be radial or unidirectional, depending on
whether the relevant cue is approximately radial or externally imposed.
\end{remark}

\begin{remark}[Relation to energy-based anisotropic variants]
The mobility-weighted Cahn--Hilliard equation
\[
\partial_t\rho=\Div(\rho\nabla q),
\qquad
q=\frac{\delta E}{\delta\rho},
\]
is a generalized gradient flow of \(E\) with respect to the
mobility-induced Onsager metric. The traction flux in
\eqref{eq:PDE},
\[
\kappa\Div\bigl(\rho D(\nu)\nabla\rho\bigr),
\qquad \kappa=\varepsilon^{-1}\eta^2,
\]
has a different interpretation: it is prescribed as an additional conservative
anisotropic transport term and is not, in general, generated by the
variational derivative of the scalar energy \eqref{eq:E}.

To compare with a genuinely energy-based modification, suppose for simplicity
that the anisotropy is represented by a scalar coefficient
\(a=a(\nu(x,t))\). If one adds the quadratic term
\[
\frac{\kappa}{2}\int_\Omega a(x,t)\rho^2\dx
\]
to the energy, then the corresponding chemical potential becomes
\[
q_{\rm eb}
=
\varepsilon^{-1}W'(\rho)-\varepsilon\Delta\rho+\kappa a(x,t)\rho .
\]
The mobility-weighted gradient-flow equation then contains the additional flux
\[
\kappa\Div\bigl(\rho\nabla(a\rho)\bigr)
=
\kappa\Div\bigl(\rho a\nabla\rho\bigr)
+
\kappa\Div\bigl(\rho^2\nabla a\bigr).
\]
Thus, even in this scalar case, the energy-based model differs from the
traction model by the drift term
\(\kappa\Div(\rho^2\nabla a)\), unless \(a\) is spatially constant. For the
tensor-valued anisotropy used in this paper, for instance
\(D(\nu)=\nu\otimes\nu\), the term
\(\Div(\rho D(\nu)\nabla\rho)\) should therefore be understood as a prescribed
anisotropic traction flux rather than as the gradient flow of a local scalar
energy.
\end{remark}

\section{Existence of weak solutions}\label{sec:analysis}

We next state the standing assumptions used in the existence proof.

\begin{assumption}\label{ass:main}
Assume:
\begin{enumerate}\itemsep0em
\item \(\Omega\subset\mathbb{R}^d\) (\(d\in\{2,3\}\)) is bounded with \(C^{1,1}\) boundary and \(T>0\).
\item \(\varepsilon>0\) and \(\eta\ge 0\) are fixed, and \(\kappa:=\varepsilon^{-1}\eta^2\).
\item \(\rho_0\in H^1(\Omega)\), \(\rho_0\ge 0\) a.e.\ in \(\Omega\), and \(E(\rho_0)<\infty\).
\item \(\nu\in L^\infty(\Omega\times(0,T);\mathbb{R}^d)\) is given and
\(D(\nu)\in L^\infty(\Omega\times(0,T);\mathbb{R}^{d\times d})\) satisfies
\(D(\nu)=D(\nu)^\top\) and for some \(\Lambda>0\),
\[
0\le \xi^\top D(\nu(x,t)) \xi \le \Lambda|\xi|^2
\quad\text{for a.e. }(x,t)\in\Omega\times(0,T),\ \forall \xi\in\mathbb{R}^d.
\]
\item \(W\in C^2(\mathbb{R})\) and there exist constants \(c_0,c_1,C_W>0\) and an exponent \(p\) satisfying
\[
2<p<\infty \quad\text{if } d=2,
\qquad
2<p\le 4 \quad\text{if } d=3,
\]
such that
\[
W(s)\ge c_0|s|^p-c_1,\qquad W''(s)\ge -c_1\qquad\forall s\in\mathbb{R},
\]
and
\[
|W'(s)|\le C_W\bigl(1+|s|^{p-1}\bigr),\qquad
|W''(s)|\le C_W\bigl(1+|s|^{p-2}\bigr)\qquad\forall s\in\mathbb{R}.
\]
\end{enumerate}
\end{assumption}

\noindent
Assumption~\ref{ass:main} implies that \(W\) is \(c_1\)-semiconvex, i.e.
\(
W''(s)\ge -c_1\) for any \(s\in\mathbb R.\)
Equivalently, the map
\(
s\mapsto W(s)+\frac{c_1}{2}s^2
\)
is convex. Hence \(W\) admits the canonical decomposition
\[
W = W_1 + W_2,
\qquad
W_1(s):=W(s)+\frac{c_1}{2}s^2 \ \text{is convex},\qquad
W_2(s):=-\frac{c_1}{2}s^2 .
\]
In particular, \(W_2''\equiv -c_1\) and \(W_1''(s)=W''(s)+c_1\ge 0\) for all \(s\in\R\).
This interpretation aligns with implicit-explicit time discretizations as often used in numerical methods for Cahn--Hilliard equations.

\begin{definition}[Weak solution]\label{def:weak}
Let Assumption~\ref{ass:main} hold.
A pair \((\rho,J)\) is called a weak solution of \eqref{eq:PDE} on \((0,T)\) if:
\begin{enumerate}\itemsep0em
\item \textbf{Regularity and nonnegativity.}
\begin{align*}
&\rho\in L^\infty(0,T;H^1(\Omega))\cap L^2(0,T;H^2(\Omega))\cap C([0,T];L^2(\Omega)),\\
&\partial_t \rho \in L^{16/9}\bigl(0,T;(W^{1,16/7}(\Omega))'\bigr),\\
&\rho\ge 0 \ \text{a.e.\ in }\Omega_T,\qquad
J\in L^{16/9}(\Omega_T;\mathbb{R}^d).
\end{align*}

\item \textbf{Chemical potential.}
There exists \(q\in L^2(0,T;L^2(\Omega))\) such that for a.e.\ \(t\in(0,T)\) and all \(\psi\in H^1(\Omega)\),
\begin{equation}\label{eq:q_weak}
\int_\Omega q(t) \psi \dx
=
\varepsilon^{-1}\int_\Omega W'(\rho(t)) \psi \dx
+\varepsilon\int_\Omega \nabla\rho(t)\cdot\nabla\psi \dx.
\end{equation}

\item \textbf{Weak mass balance and initial condition.}
For all test functions \(\psi\in C_c^\infty([0,T)\times\overline\Omega)\),
\begin{equation}\label{eq:weak_mass}
\int_0^T\!\!\int_\Omega \rho\,\partial_t\psi\,\dx\dt
\;-
\int_0^T\!\!\int_\Omega J\cdot\nabla\psi\,\dx\dt
\;+
\int_\Omega \rho_0\,\psi(\cdot,0)\,\dx
=0.
\end{equation}

\item \textbf{Diffusive flux, weighted bound, and constitutive law.}
Define the diffusive part of the flux by
\begin{equation}\label{eq:Jdiff_def}
J^{\mathrm{diff}}:=J-\kappa \rho D(\nu)\nabla\rho \qquad \text{in }\mathcal{D}'(\Omega_T).
\end{equation}
We require
\begin{equation}\label{eq:weightedJdiff}
\int_0^T  \int_\Omega \frac{|J^{\mathrm{diff}}|^2}{\rho} \dx\dt<\infty,
\qquad
\frac{|J^{\mathrm{diff}}|^2}{\rho}:=0\ \text{on }\{\rho=0\}.
\end{equation}
Moreover, for every \(\zeta\in W^{1,\infty}(\Omega;\mathbb{R}^d)\) with \(\zeta\cdot n=0\) on \(\partial\Omega\),
\begin{equation}\label{eq:flux_id}
\int_0^T  \int_\Omega J^{\mathrm{diff}}\cdot \zeta \dx\dt
=
-\int_0^T  \int_\Omega q \Div(\rho \zeta) \dx\dt.
\end{equation}
\end{enumerate}
\end{definition}

\begin{remark}\label{rem:div_rho_zeta}
If \(\zeta\in W^{1,\infty}(\Omega;\mathbb{R}^d)\), then
\[
\Div(\rho\zeta)=\zeta\cdot\nabla\rho+\rho \Div\zeta\in L^2(\Omega_T),
\]
because \(\nabla\rho\in L^2(\Omega_T)\) and
\(
\rho\in L^\infty(0,T;H^1(\Omega))\hookrightarrow L^\infty(0,T;L^6(\Omega))
\)
for \(d\le 3\). Hence the right-hand side of \eqref{eq:flux_id} is well defined.
\end{remark}

\begin{remark}[Boundary conditions incorporated in the weak formulation]
The homogeneous Neumann condition for \(\rho\) is incorporated through \eqref{eq:q_weak}.
Indeed, since \(\rho\in L^2(0,T;H^2(\Omega))\) and \(q-\varepsilon^{-1}W'(\rho)\in L^2(\Omega_T)\),
the identity \eqref{eq:q_weak} is precisely the variational formulation of
\(
-\varepsilon \Delta \rho = q-\varepsilon^{-1}W'(\rho)
\)
with homogeneous Neumann boundary condition
\(
\partial_n \rho = 0\) on \(\partial\Omega
\)
for a.e.\ \(t\in(0,T)\), in the trace sense. Likewise, the no-flux condition for the total flux
is encoded by the weak mass balance \eqref{eq:weak_mass}. In particular, all integrations by parts used
below are compatible with the stated boundary conditions and produce no boundary contributions.
\end{remark}

We now state the main analytical result.

\begin{theorem}[Global existence]\label{thm:main}
Let Assumption~\ref{ass:main} hold. Then there exists a weak solution \((\rho,J)\) in the sense
of Definition~\ref{def:weak}. In particular, \(\rho\in C([0,T];L^2(\Omega))\) and
\(\rho(0)=\rho_0\) in \(L^2(\Omega)\). Moreover:
\begin{enumerate}\itemsep0em
\item \textbf{Mass conservation.}
The map \(t\mapsto\int_\Omega\rho(t)\dx\) is continuous on \([0,T]\) (by \(\rho\in C([0,T];L^2(\Omega))\))
and
\(
\int_\Omega \rho(t) \dx=\int_\Omega \rho_0 \dx\) for every \(t\in[0,T]\).
\item \textbf{Energy inequality.} For a.e.\ \(t\in(0,T)\),
\begin{equation}\label{eq:energyineq}
E(\rho(t))
+\frac12\int_0^t  \int_\Omega \frac{|J^{\mathrm{diff}}|^2}{\rho} \dx \ds
\le
E(\rho_0)
+C(\Lambda) \kappa^2\int_0^t  \int_\Omega \rho |\nabla\rho|^2 \dx \ds.
\end{equation}
\item \textbf{Entropy inequality.}
Let \(\Phi:[0,\infty)\to\mathbb{R}\) be defined by
\begin{equation}\label{eq:Phi_def}
\Phi(s):=
\begin{cases}
s\log s - s + 1, & s>0,\\[0.2em]
1, & s=0.
\end{cases}
\end{equation}
Then for a.e.\ \(t\in(0,T)\),
\begin{equation}\label{eq:entropyineq}
\begin{aligned}
&\int_\Omega \Phi(\rho(t)) \dx
+\varepsilon\int_0^t  \int_\Omega |\Delta\rho|^2 \dx \ds
+\kappa\int_0^t  \int_\Omega \nabla\rho\cdot D(\nu)\nabla\rho \dx \ds
\\&\le
\int_\Omega \Phi(\rho_0) \dx
+\frac{c_1}{\varepsilon}\int_0^t  \int_\Omega |\nabla\rho|^2 \dx \ds.
\end{aligned}
\end{equation}
\end{enumerate}
\end{theorem}

\noindent
Since the proof of Theorem~\ref{thm:main} is rather technical, we first summarize its structure.

\begin{enumerate}[label=\bfseries Step \arabic*:, leftmargin=*, align=left]
\item \textbf{Regularization.}
We remove the mobility degeneracy and truncate the potential. This produces a uniformly parabolic approximation.

\item \textbf{Existence for fixed \((\delta,L)\) and energy inequality.}
For fixed \((\delta,L)\) we solve the regularized system by Galerkin approximation. Testing with the chemical potential yields an energy inequality; the traction cross term is controlled by Young's inequality and the bound \(0\le D(\nu)\le \Lambda I\).

\item \textbf{Entropy estimate at the \((\delta,L)\)-level and nonnegativity.}
We test with a mobility-matched entropy variable adapted to the one-sided degeneracy \(m(\rho)=\rho\).
This yields an entropy inequality controlling the fourth-order term and the anisotropic traction contribution, and it enforces nonnegativity in the limit \(\delta\downarrow0\).

\item \textbf{Passage \(\delta\downarrow 0\) for fixed \(L\).}
Using the uniform bounds and compactness, we pass to the limit \(\delta\downarrow0\) and obtain a nonnegative solution of the \(L\)-truncated system.

\item \textbf{Uniform (in \(L\)) entropy estimate.}
We introduce a truncated entropy \(\Phi_L\) associated with \(m_L\).
Since \(\Phi_L'\) is singular at \(0\), we do not test the degenerate \(L\)-system directly.
Instead, we pass to the limit \(\delta\downarrow0\) in the \((\delta,L)\)-entropy inequality and obtain an entropy estimate for the \(L\)-system.
Using mass conservation and elliptic estimates, we absorb the lower-order term generated by the concave part of \(W\), which yields a uniform control of the higher-order regularity.

\item \textbf{Uniform (in \(L\)) energy estimate.}
We combine the \(L\)-level energy inequality with the entropy bounds to absorb the indefinite traction cross term and obtain uniform-in-\(L\) energy bounds via a Grönwall argument.

\item \textbf{Passage \(L\to\infty\).}
By compactness, we extract a subsequence converging to a nonnegative limit \(\rho\).
Using strong convergence of \(m_L(\rho_L)\) in \(L^2(0,T;H^1(\Omega))\), we identify the diffusive flux and pass to the constitutive law.
The chemical potential is identified in space-time weak form via strong convergence of \(W_L'(\rho_L)\).
Finally, lower semicontinuity and the strong convergence of the traction term yield the energy and entropy inequalities, while the improved time regularity and the resulting compactness in \(C([0,T];L^2(\Omega))\) yield the initial condition.
\end{enumerate}

\subsubsection*{Step 1: Regularization} 
For \(\delta\in(0,1)\) and \(L>1\) define
\[
m_{\delta,L}(r):=\min\{\max\{\delta,r\},L\},\qquad r\in\mathbb{R}.
\]
Then \(\delta\le m_{\delta,L}(r)\le L\) for all \(r\in\mathbb{R}\) and
\(m_{\delta,L}(r)\to m_L(r):=\min\{r_+,L\}\) as \(\delta\downarrow 0\).
Furthermore, we choose \(W_L\in C^2(\mathbb{R})\) such that the following lemma holds.

\begin{lemma}[Truncation of the potential]\label{lem:WL}
Assume \(W\in C^2(\mathbb{R})\) satisfies Assumption~\ref{ass:main}(5).
Then there exists a family \((W_L)_{L>1}\subset C^2(\mathbb{R})\) such that for every \(L>1\):
\begin{enumerate}\itemsep0em
\item \(W_L(s)=W(s)\) and \(W_L'(s)=W'(s)\) for all \(|s|\le L\).
\item \(W_L''(s)\ge -c_1\) for all \(s\in\mathbb{R}\).
\item \emph{Uniform polynomial growth.}
There exists \(C>0\), independent of \(L\), such that
\[
|W_L'(s)|\le C\bigl(1+|s|^{p-1}\bigr),\qquad
|W_L''(s)|\le C\bigl(1+|s|^{p-2}\bigr)\qquad(\forall s\in\mathbb{R}).
\]
\item \emph{Coercivity.}
There exist \(c_2,c_3>0\), independent of \(L\), such that
\[
W_L(s)\ge c_2|s|^p-c_3\qquad(\forall s\in\mathbb{R}).
\]
\item \emph{Consistency.}
\(W_L\to W\) and \(W_L'\to W'\) locally uniformly on \(\mathbb{R}\) as \(L\to\infty\).
\end{enumerate}
\end{lemma}

\begin{proof}
A standard construction modifies \(W\) smoothly outside \([-L,L]\) by attaching a polynomial with matching values and
first derivatives at \(\pm L\), and then mollifying near \(\pm L\) to retain \(C^2\) regularity and the lower bound on \(W_L''\).
Uniform growth bounds follow from the polynomial-growth assumptions on \(W'\) and \(W''\).
\end{proof}

Then, for fixed \((\delta,L)\), we consider the regularized system
\begin{align}
\partial_t\rho_{\delta,L}
&=\Div\bigl(m_{\delta,L}(\rho_{\delta,L})\nabla q_{\delta,L}\bigr)
+\kappa \Div\bigl(m_{\delta,L}(\rho_{\delta,L})D(\nu)\nabla\rho_{\delta,L}\bigr),
\label{eq:regPDE_deltaL}\\
q_{\delta,L}
&=\varepsilon^{-1}W_L'(\rho_{\delta,L})-\varepsilon\Delta\rho_{\delta,L},
\label{eq:regq_deltaL}
\end{align}
with boundary conditions
\[
\nabla\rho_{\delta,L}\cdot n = 0,
\qquad
\bigl(m_{\delta,L}(\rho_{\delta,L})\nabla q_{\delta,L}
+\kappa m_{\delta,L}(\rho_{\delta,L})D(\nu)\nabla\rho_{\delta,L}\bigr)\cdot n = 0
\quad\text{on }\partial\Omega\times(0,T),
\]
and initial condition \(\rho_{\delta,L}(0)=\rho_0\).
We also set the associated energy
\[
E_L(\rho):=\int_\Omega\Bigl(\varepsilon^{-1}W_L(\rho)+\frac{\varepsilon}{2}|\nabla\rho|^2\Bigr) \dx.
\]

\subsubsection*{Step 2: Existence for fixed \((\delta,L)\) and energy inequality}

\begin{lemma}\label{lem:exist_deltaL}
For every \(\delta\in(0,1)\) and \(L>1\) there exists a weak solution \((\rho_{\delta,L},q_{\delta,L})\) to \eqref{eq:regPDE_deltaL}--\eqref{eq:regq_deltaL} with
\[
\rho_{\delta,L}\in H^1(0,T;H^{-1}(\Omega))\cap L^\infty(0,T;H^1(\Omega))\cap L^2(0,T;H^2(\Omega)),\qquad
q_{\delta,L}\in L^2(0,T;H^1(\Omega)),
\]
satisfying mass conservation and, for a.e.\ \(t\in(0,T)\),
\begin{equation}\label{eq:energy_deltaL}
\begin{aligned}
&E_L(\rho_{\delta,L}(t))
+\frac12\int_0^t  \int_\Omega m_{\delta,L}(\rho_{\delta,L}) |\nabla q_{\delta,L}|^2 \dx\ds
\\ &\le
E_L(\rho_0)
+C(\Lambda)\kappa^2\int_0^t  \int_\Omega m_{\delta,L}(\rho_{\delta,L}) |\nabla\rho_{\delta,L}|^2 \dx\ds.
\end{aligned}
\end{equation}
\end{lemma}

\begin{proof}
We proceed by a Galerkin approximation, derive estimates that are uniform in the Galerkin dimension,
and then pass to the limit. \smallskip

\noindent\textit{Galerkin approximation.} Let \(\{w_k\}_{k\ge1}\) be the eigenfunctions of the Neumann--Laplacian on \(\Omega\), normalized so that
they form an orthonormal basis of \(L^2(\Omega)\) and satisfy
\[
-\Delta w_k=\lambda_k w_k \quad\text{in }\Omega,\qquad \nabla w_k\cdot n=0 \quad\text{on }\partial\Omega.
\]
We recall that \(w_1\) is constant and that \(\{w_k\}_{k\ge1}\) is also orthogonal in \(H^1(\Omega)\).
For \(N\in\mathbb N\) define
\[
V_N:=\mathrm{span}\{w_1,\dots,w_N\}\subset H^1(\Omega).
\]
We seek approximate solutions
\(\rho_{\delta,L,N},\,q_{\delta,L,N}\in C^1([0,T_N);V_N)\) such that
for all test functions \(\varphi,\psi\in V_N\) and a.e.\ \(t\in(0,T_N)\),
\begin{subequations}\label{eq:Galerkin_deltaL}
\begin{align}
(\partial_t\rho_{\delta,L,N},\varphi)
&+\bigl(m_{\delta,L}(\rho_{\delta,L,N})\nabla q_{\delta,L,N},\nabla\varphi\bigr)
+\kappa\bigl(m_{\delta,L}(\rho_{\delta,L,N})D(\nu)\nabla\rho_{\delta,L,N},\nabla\varphi\bigr)=0,
\label{eq:G1_deltaL_full}
\\
(q_{\delta,L,N},\psi)
&=\varepsilon^{-1}\bigl(W_L'(\rho_{\delta,L,N}),\psi\bigr)
+\varepsilon\bigl(\nabla\rho_{\delta,L,N},\nabla\psi\bigr).
\label{eq:G2_deltaL_full}
\end{align}
\end{subequations}
The initial value is chosen as \(\rho_{\delta,L,N}(0)=P_N\rho_0\), where \(P_N\) denotes the \(L^2\)-orthogonal projection onto \(V_N\).

Since \(V_N\) is finite-dimensional, \eqref{eq:G2_deltaL_full} determines \(q_{\delta,L,N}\) uniquely from
\(\rho_{\delta,L,N}\) at every time. Indeed, if
\[
\rho_{\delta,L,N}(t)=\sum_{i=1}^N a_i(t) w_i,
\qquad
q_{\delta,L,N}(t)=\sum_{i=1}^N b_i(t) w_i,
\]
then \eqref{eq:G2_deltaL_full} is a linear system for the coefficients \(b_i(t)\) with invertible mass matrix.
Substituting this relation into \eqref{eq:G1_deltaL_full} yields an ODE system for the vector
\(a(t)=(a_1(t),\dots,a_N(t))\).
Now \(m_{\delta,L}\) is globally Lipschitz and bounded, and \(W_L'\) is locally Lipschitz.
Therefore the right-hand side of the ODE system is locally Lipschitz on \(\mathbb R^N\).
By the Cauchy--Lipschitz theorem, there exists a unique maximal solution on some interval
\([0,T_N)\) with \(0<T_N\le T\). \smallskip

\noindent\textit{Uniform estimates.} We next derive the energy inequality at the Galerkin level. Since \(w_1\) is constant, testing \eqref{eq:G1_deltaL_full} with \(\varphi=w_1\) gives
\[
(\partial_t\rho_{\delta,L,N},w_1)=0,
\]
and hence
\begin{equation}\label{eq:mass_Galerkin_deltaL}
\int_\Omega \rho_{\delta,L,N}(t)\dx
=
\int_\Omega P_N\rho_0\dx
=
\int_\Omega \rho_0\dx
\qquad\text{for all }t\in[0,T_N).
\end{equation}
Moreover, since \(q_{\delta,L,N}\in V_N\), we may choose \(\varphi=q_{\delta,L,N}\) in \eqref{eq:G1_deltaL_full}, which gives
\begin{equation}\label{eq:Galerkin_test_q}
(\partial_t\rho_{\delta,L,N},q_{\delta,L,N})
+\!\int_\Omega\!\! m_{\delta,L}(\rho_{\delta,L,N})|\nabla q_{\delta,L,N}|^2\dx
+\kappa\!\int_\Omega\!\! m_{\delta,L}(\rho_{\delta,L,N})D(\nu)\nabla\rho_{\delta,L,N}\cdot\nabla q_{\delta,L,N}\dx
=0.
\end{equation}
Since also \(\partial_t\rho_{\delta,L,N}\in V_N\), we may choose \(\psi=\partial_t\rho_{\delta,L,N}\) in
\eqref{eq:G2_deltaL_full}. This yields
\[
\begin{aligned}
(q_{\delta,L,N},\partial_t\rho_{\delta,L,N})
&=
\varepsilon^{-1}\bigl(W_L'(\rho_{\delta,L,N}),\partial_t\rho_{\delta,L,N}\bigr)
+\varepsilon\bigl(\nabla\rho_{\delta,L,N},\nabla\partial_t\rho_{\delta,L,N}\bigr)
\\
&=
\ddt E_L(\rho_{\delta,L,N}).
\end{aligned}
\]
Combining this identity with \eqref{eq:Galerkin_test_q}, we arrive at
\[
\ddt E_L(\rho_{\delta,L,N})
+\int_\Omega m_{\delta,L}(\rho_{\delta,L,N})|\nabla q_{\delta,L,N}|^2\dx
=
-\kappa\int_\Omega m_{\delta,L}(\rho_{\delta,L,N})D(\nu)\nabla\rho_{\delta,L,N}\cdot\nabla q_{\delta,L,N}\dx.
\]
Using \(0\le D(\nu)\le \Lambda I\) a.e.\ and Young's inequality with parameter \(1/2\),
\[
\begin{aligned}
\kappa\Bigl|\int_\Omega m_{\delta,L}(\rho_{\delta,L,N})
D(\nu)\nabla\rho_{\delta,L,N}\cdot\nabla q_{\delta,L,N}\dx\Bigr|
&\le
\kappa\Lambda\int_\Omega m_{\delta,L}(\rho_{\delta,L,N})
|\nabla\rho_{\delta,L,N}| |\nabla q_{\delta,L,N}|\dx
\\
&\le
\frac12\int_\Omega m_{\delta,L}(\rho_{\delta,L,N})|\nabla q_{\delta,L,N}|^2\dx
\\&\quad
+\frac{\Lambda^2}{2}\kappa^2\int_\Omega m_{\delta,L}(\rho_{\delta,L,N})|\nabla\rho_{\delta,L,N}|^2\dx.
\end{aligned}
\]
Setting \(C(\Lambda):=\Lambda^2/2\), we obtain
\begin{equation}\label{eq:energy_Galerkin_deltaL}
\ddt E_L(\rho_{\delta,L,N})
+\frac12\int_\Omega m_{\delta,L}(\rho_{\delta,L,N})|\nabla q_{\delta,L,N}|^2\dx
\le
C(\Lambda)\kappa^2\int_\Omega m_{\delta,L}(\rho_{\delta,L,N})|\nabla\rho_{\delta,L,N}|^2\dx.
\end{equation}
To control the right-hand side, note that \(m_{\delta,L}(r)\le L\), and by the coercivity of \(W_L\),
\[
E_L(\rho_{\delta,L,N})
=
\int_\Omega \Bigl(\varepsilon^{-1}W_L(\rho_{\delta,L,N})
+\frac{\varepsilon}{2}|\nabla\rho_{\delta,L,N}|^2\Bigr)\dx
\ge
\frac{\varepsilon}{2}\|\nabla\rho_{\delta,L,N}\|_{L^2(\Omega)}^2-C.
\]
Hence
\[
\int_\Omega m_{\delta,L}(\rho_{\delta,L,N})|\nabla\rho_{\delta,L,N}|^2\dx
\le
L\|\nabla\rho_{\delta,L,N}\|_{L^2(\Omega)}^2
\le C_L\bigl(1+E_L(\rho_{\delta,L,N})\bigr).
\]
Inserting this into \eqref{eq:energy_Galerkin_deltaL}, we obtain
\[
\ddt \bigl(1+E_L(\rho_{\delta,L,N})\bigr)
\le C_{L}\bigl(1+E_L(\rho_{\delta,L,N})\bigr).
\]
By Grönwall's lemma,
\begin{equation}\label{eq:energy_uniformN_deltaL}
\sup_{t\in[0,T_N)}E_L(\rho_{\delta,L,N}(t))
\le C_{T,L}\bigl(1+E_L(P_N\rho_0)\bigr),
\end{equation}
where the constant \(C_{T,L}\) is independent of \(N\) and \(\delta\). Since \(P_N\rho_0\to \rho_0\) strongly in \(H^1(\Omega)\), and \(H^1(\Omega)\hookrightarrow L^p(\Omega)\) while \(W_L\) has polynomial growth of order \(p\), we have
\(
E_L(P_N\rho_0)\to E_L(\rho_0)
\).
In particular, \(E_L(P_N\rho_0)\le C_L\) uniformly in \(N\). Thus \((\rho_{\delta,L,N})_N\) is uniformly bounded in \(L^\infty(0,T_N;H^1(\Omega))\), and integrating
\eqref{eq:energy_Galerkin_deltaL} in time gives
\begin{equation}\label{eq:diss_uniformN_deltaL}
\int_0^{T_N}\!\!\int_\Omega m_{\delta,L}(\rho_{\delta,L,N})|\nabla q_{\delta,L,N}|^2\dx\dt
\le C_{T,\delta,L}.
\end{equation}
Since \(m_{\delta,L}\ge\delta\), this implies
\begin{equation}\label{eq:q_grad_uniformN_deltaL}
\int_0^{T_N}\!\!\int_\Omega |\nabla q_{\delta,L,N}|^2\dx\dt
\le \delta^{-1} C_{T,\delta,L}.
\end{equation}

The ODE system is finite-dimensional and its vector field is locally Lipschitz. The only possible obstruction to continuation would be blow-up of the coefficient vector. But \eqref{eq:energy_uniformN_deltaL} yields a uniform bound in \(H^1(\Omega)\); hence all coefficients remain bounded on every compact subinterval of \([0,T)\). Therefore no blow-up occurs, and the local solution extends globally, i.e. \(T_N=T\). \smallskip

\noindent\textit{Additional estimates.}  Since \(\rho_{\delta,L,N}(t)\in V_N\) and \(V_N\) is spanned by Neumann eigenfunctions, also
\(-\Delta \rho_{\delta,L,N}(t)\in V_N\). Thus we may test \eqref{eq:G2_deltaL_full} with
\(\psi=-\Delta\rho_{\delta,L,N}(t)\) and obtain
\[
(q_{\delta,L,N},-\Delta\rho_{\delta,L,N})
=
\varepsilon^{-1}(W_L'(\rho_{\delta,L,N}),-\Delta\rho_{\delta,L,N})
+\varepsilon(\nabla\rho_{\delta,L,N},\nabla(-\Delta\rho_{\delta,L,N})),
\]
which gives after integration by parts
\[
\varepsilon\|\Delta\rho_{\delta,L,N}\|_{L^2(\Omega)}^2
+\varepsilon^{-1}\int_\Omega W_L''(\rho_{\delta,L,N})|\nabla\rho_{\delta,L,N}|^2\dx
=
(\nabla q_{\delta,L,N},\nabla\rho_{\delta,L,N}).
\]
Using \(W_L''\ge -c_1\), we infer
\[
\varepsilon\|\Delta\rho_{\delta,L,N}\|_{L^2(\Omega)}^2
\le
\|\nabla q_{\delta,L,N}\|_{L^2(\Omega)}\|\nabla\rho_{\delta,L,N}\|_{L^2(\Omega)}
+\frac{c_1}{\varepsilon}\|\nabla\rho_{\delta,L,N}\|_{L^2(\Omega)}^2.
\]
Integrating in time over \((0,T)\), and using \eqref{eq:q_grad_uniformN_deltaL} together with
\eqref{eq:energy_uniformN_deltaL}, we conclude
\begin{equation}\label{eq:H2_uniformN_deltaL}
\|\Delta\rho_{\delta,L,N}\|_{L^2(0,T;L^2(\Omega))}\le C_{T,\delta,L}.
\end{equation}
By elliptic regularity for the Neumann Laplacian and mass conservation
(which controls the mean value), this implies
\[
\|\rho_{\delta,L,N}\|_{L^2(0,T;H^2(\Omega))}\le C_{T,\delta,L}.
\]

We now estimate \(\partial_t\rho_{\delta,L,N}\) in \(L^2(0,T;H^{-1}(\Omega))\).
Let \(\varphi\in H^1(\Omega)\). Since \(\partial_t\rho_{\delta,L,N}\in V_N\), we have
\[
\langle \partial_t\rho_{\delta,L,N},\varphi\rangle_{H^{-1},H^1}
=
(\partial_t\rho_{\delta,L,N},P_N\varphi).
\]
Testing \eqref{eq:G1_deltaL_full} with \(P_N\varphi\) gives
\[
\begin{aligned}
\bigl|\langle \partial_t\rho_{\delta,L,N},\varphi\rangle\bigr|
&\le
\Bigl|\int_\Omega m_{\delta,L}(\rho_{\delta,L,N})\nabla q_{\delta,L,N}\cdot\nabla P_N\varphi\dx\Bigr|
\\&\quad
+\kappa\Bigl|\int_\Omega m_{\delta,L}(\rho_{\delta,L,N})D(\nu)\nabla\rho_{\delta,L,N}\cdot\nabla P_N\varphi\dx\Bigr|.
\end{aligned}
\]
Using \(m_{\delta,L}\le L\), \(0\le D(\nu)\le \Lambda I\), and the \(H^1\)-stability of the spectral projection,
we deduce
\[
\bigl|\langle \partial_t\rho_{\delta,L,N},\varphi\rangle\bigr|
\le
CL\|\nabla q_{\delta,L,N}\|_{L^2(\Omega)}\|\nabla\varphi\|_{L^2(\Omega)}
+
C\kappa L\Lambda \|\nabla\rho_{\delta,L,N}\|_{L^2(\Omega)}\|\nabla\varphi\|_{L^2(\Omega)}.
\]
Therefore
\[
\|\partial_t\rho_{\delta,L,N}\|_{H^{-1}(\Omega)}
\le
C\|\nabla q_{\delta,L,N}\|_{L^2(\Omega)}
+
C\|\nabla\rho_{\delta,L,N}\|_{L^2(\Omega)}.
\]
Combining this with \eqref{eq:q_grad_uniformN_deltaL} and \eqref{eq:energy_uniformN_deltaL}, we obtain
\begin{equation}\label{eq:dt_uniformN_deltaL}
\|\partial_t\rho_{\delta,L,N}\|_{L^2(0,T;H^{-1}(\Omega))}\le C_{T,\delta,L}.
\end{equation}

It remains to obtain a uniform bound for \(q_{\delta,L,N}\) in \(L^2(0,T;H^1(\Omega))\).
Testing \eqref{eq:G2_deltaL_full} with the constant function \(\psi\equiv1\in V_N\) yields
\[
\int_\Omega q_{\delta,L,N}(t)\dx
=
\varepsilon^{-1}\int_\Omega W_L'(\rho_{\delta,L,N}(t))\dx
\qquad\text{for a.e. }t\in(0,T).
\]
Thus, writing \(\overline q_{\delta,L,N}(t):=|\Omega|^{-1}\int_\Omega q_{\delta,L,N}(t)\dx\), we have
\[
|\overline q_{\delta,L,N}(t)|
\le
C\|W_L'(\rho_{\delta,L,N}(t))\|_{L^1(\Omega)}
\le
C\|W_L'(\rho_{\delta,L,N}(t))\|_{L^2(\Omega)}.
\]
Using Lemma~\ref{lem:WL}(3) once more and the uniform \(L^\infty(0,T;L^{2(p-1)}(\Omega))\)-bound on
\(\rho_{\delta,L,N}\), we obtain
\[
\|W_L'(\rho_{\delta,L,N})\|_{L^2(\Omega_T)}\le C_{T,\delta,L},
\]
hence
\[
\|\overline q_{\delta,L,N}\|_{L^2(0,T)}\le C_{T,\delta,L}.
\]
By Poincar\'e's inequality,
\[
\|q_{\delta,L,N}(t)-\overline q_{\delta,L,N}(t)\|_{L^2(\Omega)}
\le C\|\nabla q_{\delta,L,N}(t)\|_{L^2(\Omega)}.
\]
Therefore, together with \eqref{eq:q_grad_uniformN_deltaL},
\begin{equation} \label{eq:q_uniformN_deltaL}
\|q_{\delta,L,N}\|_{L^2(0,T;H^1(\Omega))}
\le
C\Bigl(
\|\nabla q_{\delta,L,N}\|_{L^2(\Omega_T)}
+\|\overline q_{\delta,L,N}\|_{L^2(0,T)}
\Bigr)
\le C_{T,\delta,L}.
\end{equation}

\noindent\textit{Limit process.} The estimates \eqref{eq:energy_uniformN_deltaL}, \eqref{eq:H2_uniformN_deltaL}, \eqref{eq:dt_uniformN_deltaL}, and \eqref{eq:q_uniformN_deltaL}
show that, up to a subsequence,
\[
\begin{aligned}
\rho_{\delta,L,N}&\rightharpoonup^\ast \rho_{\delta,L}
&&\quad\text{in }L^\infty(0,T;H^1(\Omega)), \\
\rho_{\delta,L,N}&\rightharpoonup \rho_{\delta,L}
&&\quad\text{in }L^2(0,T;H^2(\Omega)), \\
\partial_t\rho_{\delta,L,N}&\rightharpoonup \partial_t\rho_{\delta,L}
&&\quad\text{in }L^2(0,T;H^{-1}(\Omega)), \\
q_{\delta,L,N}&\rightharpoonup q_{\delta,L}
&&\quad\text{in }L^2(0,T;H^1(\Omega)),
\end{aligned}
\]
Moreover,
the Aubin--Lions lemma \cite{simon1986compact} yields
\begin{equation}\label{eq:rho_strong_deltaL}
\rho_{\delta,L,N}\to \rho_{\delta,L}
\qquad\text{strongly in }L^2(0,T;H^1(\Omega)).
\end{equation}
Since \(m_{\delta,L}\) is globally Lipschitz, this implies
\[
m_{\delta,L}(\rho_{\delta,L,N})\to m_{\delta,L}(\rho_{\delta,L})
\qquad\text{strongly in }L^2(\Omega_T).
\]
We next show the strong convergence of \(W_L'(\rho_{\delta,L,N})\).
Because \(d\le3\) and \(2(p-1)\le 6\) if \(d=3\) (while \(H^1(\Omega)\hookrightarrow L^r(\Omega)\) for every finite \(r\) if \(d=2\)),
the embedding \(H^1(\Omega)\hookrightarrow L^{2(p-1)}(\Omega)\) is continuous.
Hence \eqref{eq:rho_strong_deltaL} yields
\[
\rho_{\delta,L,N}\to \rho_{\delta,L}
\qquad\text{strongly in }L^2\bigl(0,T;L^{2(p-1)}(\Omega)\bigr).
\]
Using Lemma~\ref{lem:WL}(3) and the mean-value theorem, we obtain for all \(a,b\in\mathbb R\),
\[
|W_L'(a)-W_L'(b)|
\le C\bigl(1+|a|^{p-2}+|b|^{p-2}\bigr)|a-b|,
\]
with \(C>0\) independent of \(N\). Therefore, by Hölder's inequality,
\[
\begin{aligned}
&\|W_L'(\rho_{\delta,L,N})-W_L'(\rho_{\delta,L})\|_{L^2(\Omega_T)}
\\
&\le
C\Bigl\|1+|\rho_{\delta,L,N}|^{p-2}+|\rho_{\delta,L}|^{p-2}\Bigr\|_{L^\infty\!(0,T;L^{\frac{2(p-1)}{p-2}}(\Omega))}
\|\rho_{\delta,L,N}-\rho_{\delta,L}\|_{L^2\!(0,T;L^{2(p-1)}(\Omega))}
\to0.
\end{aligned}
\]
Hence
\[
W_L'(\rho_{\delta,L,N})\to W_L'(\rho_{\delta,L})
\qquad\text{strongly in }L^2(\Omega_T).
\]

We may now pass to the limit in the Galerkin formulation.
The linear terms pass by weak convergence, and the nonlinear terms pass by the strong convergence of
\(m_{\delta,L}(\rho_{\delta,L,N})\) and \(W_L'(\rho_{\delta,L,N})\).
Thus \((\rho_{\delta,L},q_{\delta,L})\) satisfies the weak form of
\eqref{eq:regPDE_deltaL}--\eqref{eq:regq_deltaL}, with
\[
\rho_{\delta,L}\in H^1(0,T;H^{-1}(\Omega))\cap L^\infty(0,T;H^1(\Omega))\cap L^2(0,T;H^2(\Omega)),
\qquad
q_{\delta,L}\in L^2(0,T;H^1(\Omega)).
\]
Mass conservation follows by passing to the limit in \eqref{eq:mass_Galerkin_deltaL}, and the energy inequality
\eqref{eq:energy_deltaL} follows from \eqref{eq:energy_Galerkin_deltaL} by weak lower semicontinuity.
\end{proof}

\subsubsection*{Step 3: Entropy estimate at the \((\delta,L)\)-level and nonnegativity}

Define \(\Phi_{\delta,L}\in C^{1,1}(\mathbb R)\) by
\[
\Phi_{\delta,L}''(r)=\frac1{m_{\delta,L}(r)},
\qquad
\Phi_{\delta,L}(1)=\Phi_{\delta,L}'(1)=0.
\]
Since \(\delta\le m_{\delta,L}\le L\), we have \(0\le \Phi_{\delta,L}''\le \delta^{-1}\), hence \(\Phi_{\delta,L}'\) is globally Lipschitz. Therefore, by the Sobolev chain rule,
\[
\nabla\Phi_{\delta,L}'(\rho_{\delta,L})
=
\Phi_{\delta,L}''(\rho_{\delta,L})\nabla\rho_{\delta,L}
\in L^2(\Omega_T),
\]
because \(0\le \Phi_{\delta,L}''\le \delta^{-1}\) and \(\nabla\rho_{\delta,L}\in L^2(\Omega_T)\). Consequently,
\(
\Phi_{\delta,L}'(\rho_{\delta,L})\in L^2(0,T;H^1(\Omega)).
\)
Moreover, a standard chain-rule argument for convex/Lipschitz functionals along
trajectories in \(H^1(0,T;H^{-1}(\Omega))\cap L^2(0,T;H^1(\Omega))\) yields
(see, e.g., \cite{colli_visintin_1990} and, in a closely related Cahn--Hilliard
context, \cite[Lemma~4.1]{giorgini_grasselli_wu_2018})
\[
\ddt\int_\Omega \Phi_{\delta,L}(\rho_{\delta,L}(t)) \dx
=
\bigl\langle \partial_t\rho_{\delta,L}(t),\Phi_{\delta,L}'(\rho_{\delta,L}(t))\bigr\rangle_{H^{-1},H^1}
\]
in \(\mathcal D'(0,T)\). Thus, testing \eqref{eq:regPDE_deltaL} in weak form with
\(\varphi=\Phi_{\delta,L}'(\rho_{\delta,L})\) and using
\(m_{\delta,L}\Phi_{\delta,L}''\equiv 1\) yields
\begin{equation}\label{eq:entropy_id_deltaL}
\ddt\int_\Omega \Phi_{\delta,L}(\rho_{\delta,L}) \dx
+\int_\Omega \nabla q_{\delta,L}\cdot\nabla\rho_{\delta,L} \dx
+\kappa\int_\Omega \nabla\rho_{\delta,L}\cdot D(\nu)\nabla\rho_{\delta,L} \dx
=0.
\end{equation}
Using \(q_{\delta,L}=\varepsilon^{-1}W_L'(\rho_{\delta,L})-\varepsilon\Delta\rho_{\delta,L}\) and integrating by parts gives
\[
\begin{aligned}
\int_\Omega \nabla q_{\delta,L}\cdot\nabla\rho_{\delta,L} \dx
&=
\varepsilon\int_\Omega |\Delta\rho_{\delta,L}|^2 \dx
+\varepsilon^{-1}\int_\Omega W_L''(\rho_{\delta,L})|\nabla\rho_{\delta,L}|^2 \dx \\
&\ge
\varepsilon\int_\Omega |\Delta\rho_{\delta,L}|^2 \dx
-\frac{c_1}{\varepsilon}\int_\Omega |\nabla\rho_{\delta,L}|^2 \dx.
\end{aligned}
\]
Hence,
\begin{equation}\label{eq:entropy_deltaL}
\begin{aligned}
&\ddt\int_\Omega \Phi_{\delta,L}(\rho_{\delta,L}) \dx
+\varepsilon\int_\Omega |\Delta\rho_{\delta,L}|^2 \dx
+\kappa\int_\Omega \nabla\rho_{\delta,L}\cdot D(\nu)\nabla\rho_{\delta,L} \dx
\\ &\le
\frac{c_1}{\varepsilon}\int_\Omega |\nabla\rho_{\delta,L}|^2 \dx.
\end{aligned}
\end{equation}
We now extract from \eqref{eq:entropy_deltaL} a bound on the negative part of \(\rho_{\delta,L}\).
Since \(\Phi_{\delta,L}''\ge 0\) and
\(\Phi_{\delta,L}'(1)=0\), the function
\(\Phi_{\delta,L}'\) is non-decreasing. Hence
\[
\Phi_{\delta,L}'(0)\le \Phi_{\delta,L}'(1)=0.
\]
Moreover, for all \(r\le 0\) one has
\(\Phi_{\delta,L}''(r)=\delta^{-1}\). Therefore Taylor's formula at \(r=0\) gives
\[
\Phi_{\delta,L}(r)
=
\Phi_{\delta,L}(0)+\Phi_{\delta,L}'(0)r+\frac{1}{2\delta}r^2
\qquad (r\le0).
\]
Since \(r\le0\) and \(\Phi_{\delta,L}'(0)\le0\), the linear term satisfies
\[
\Phi_{\delta,L}'(0)r\ge0.
\]
Furthermore, \(\Phi_{\delta,L}\) is convex and normalized by
\(\Phi_{\delta,L}(1)=0\), \(\Phi_{\delta,L}'(1)=0\). Thus \(1\) is a global minimizer and
\(\Phi_{\delta,L}\ge0\) on \(\mathbb R\). Consequently,
\[
\Phi_{\delta,L}(r)
\ge \frac{1}{2\delta}r^2
\qquad (r\le0).
\]
For \(r\ge0\), the estimate \(r_-^2\le 2\delta\Phi_{\delta,L}(r)\) is trivial because
\(r_-=0\) and \(\Phi_{\delta,L}\ge0\). Hence
\[
r_-^2\le 2\delta \Phi_{\delta,L}(r)
\qquad\forall r\in\mathbb R.
\]
and therefore
\[
\|\bigl(\rho_{\delta,L}(t)\bigr)_-\|_{L^2(\Omega)}^2
\le 2\delta\int_\Omega \Phi_{\delta,L}(\rho_{\delta,L}(t)) \dx.
\]
In particular, along any sequence \(\delta\downarrow 0\) for fixed \(L\), the negative part
converges to \(0\) in \(L^\infty(0,T;L^2(\Omega))\), hence any limit \(\rho_L\) satisfies
\(\rho_L\ge 0\) a.e.

\subsubsection*{Step 4: Passage \(\delta\downarrow 0\) for fixed \(L>1\)}

\begin{lemma}[Uniform time-derivative bound for fixed \(L\)]\label{lem:dt_fixedL_uniform_delta}
There exists a constant \(C_L>0\), independent of \(\delta\in(0,1)\), such that
\[
\|\partial_t\rho_{\delta,L}\|_{L^2(0,T;H^{-1}(\Omega))}\le C_L.
\]
\end{lemma}

\begin{proof}
Define the total flux
\[
J_{\delta,L}:=
m_{\delta,L}(\rho_{\delta,L})\nabla q_{\delta,L}
+\kappa m_{\delta,L}(\rho_{\delta,L})D(\nu)\nabla\rho_{\delta,L}.
\]
Then \eqref{eq:regPDE_deltaL} reads
\(
\partial_t\rho_{\delta,L}=\Div J_{\delta,L}\) in \(L^2(0,T;H^{-1}(\Omega)).
\)
By \eqref{eq:energy_deltaL} and the bound \(m_{\delta,L}\le L\),
\[
\begin{aligned}
\|m_{\delta,L}(\rho_{\delta,L})\nabla q_{\delta,L}\|_{L^2(\Omega_T)}^2
&=
\int_{\Omega_T} m_{\delta,L}(\rho_{\delta,L})^2 |\nabla q_{\delta,L}|^2 \dx\dt
\\
&\le
L\int_{\Omega_T} m_{\delta,L}(\rho_{\delta,L}) |\nabla q_{\delta,L}|^2 \dx\dt
\le C_L .
\end{aligned}
\]
Moreover, using \(m_{\delta,L}\le L\), \(0\le D(\nu)\le \Lambda I\), and the
\(L^\infty(0,T;H^1(\Omega))\)-bound from Lemma~\ref{lem:exist_deltaL},
\[
\|m_{\delta,L}(\rho_{\delta,L})D(\nu)\nabla\rho_{\delta,L}\|_{L^2(\Omega_T)}^2
\le
\Lambda^2L^2\|\nabla\rho_{\delta,L}\|_{L^2(\Omega_T)}^2
\le C_L .
\]
Hence \(\|J_{\delta,L}\|_{L^2(\Omega_T)}\le C_L\).
Now let \(\varphi\in H^1(\Omega)\). Then
\[
\bigl|\langle \partial_t\rho_{\delta,L},\varphi\rangle_{H^{-1},H^1}\bigr|
=
\bigl|\langle \Div J_{\delta,L},\varphi\rangle_{H^{-1},H^1}\bigr|
=
\left|\int_\Omega J_{\delta,L}\cdot\nabla\varphi\,\dx\right|
\le
\|J_{\delta,L}\|_{L^2(\Omega)}\|\nabla\varphi\|_{L^2(\Omega)}.
\]
Therefore
\[
\|\partial_t\rho_{\delta,L}\|_{H^{-1}(\Omega)}
\le
\|J_{\delta,L}\|_{L^2(\Omega)}.
\]
Squaring and integrating in time yields the claim.
\end{proof}

Thus, by Lemma~\ref{lem:exist_deltaL}, \eqref{eq:entropy_deltaL}, and
Lemma~\ref{lem:dt_fixedL_uniform_delta}, the family \((\rho_{\delta,L})_{\delta\in(0,1)}\)
is bounded uniformly in \(\delta\) in
\[
L^\infty(0,T;H^1(\Omega))
\cap H^1(0,T;H^{-1}(\Omega))
\cap L^2(0,T;H^2(\Omega)).
\]
Hence, along a subsequence \(\delta\downarrow0\),
\[
\begin{aligned}
\rho_{\delta,L}&\rightharpoonup^\ast \rho_L  &&\text{in }L^\infty(0,T;H^1(\Omega)), \\
\rho_{\delta,L}&\rightharpoonup \rho_L  &&\text{in }L^2(0,T;H^2(\Omega)), \\
\partial_t\rho_{\delta,L}&\rightharpoonup \partial_t\rho_L  &&\text{in }L^2(0,T;H^{-1}(\Omega)).
\end{aligned}
\]
By the Aubin--Lions compactness lemma \cite{simon1986compact},
\[
\rho_{\delta,L}\to \rho_L
\qquad\text{strongly in }L^2(0,T;H^1(\Omega))\cap C([0,T];L^2(\Omega)).
\]
Thus,
\[
m_{\delta,L}(\rho_{\delta,L})\to m_L(\rho_L):=\min\{\rho_L,L\}
\qquad\text{strongly in }L^2(\Omega_T),
\]
and \(\rho_L\ge0\) a.e.
Since \(L\) is fixed, the growth bound in Lemma~\ref{lem:WL}(3), together with
\(\rho_{\delta,L}\in L^\infty(0,T;H^1(\Omega))\), implies
\[
\|W_L'(\rho_{\delta,L})\|_{L^2(\Omega_T)}\le C_L.
\]
Since also \(\Delta\rho_{\delta,L}\) is bounded in \(L^2(\Omega_T)\) by \eqref{eq:entropy_deltaL}, the identity
\[
q_{\delta,L}=\varepsilon^{-1}W_L'(\rho_{\delta,L})-\varepsilon\Delta\rho_{\delta,L}
\]
yields
\[
\|q_{\delta,L}\|_{L^2(\Omega_T)}\le C_L.
\]
Hence, after extraction,
\[
q_{\delta,L}\rightharpoonup q_L
\qquad\text{weakly in }L^2(\Omega_T).
\]
Moreover, the strong convergence of \(\rho_{\delta,L}\) in \(L^2(0,T;H^1(\Omega))\) implies,
exactly as in the Galerkin limit of Step~2, that
\[
W_L'(\rho_{\delta,L})\to W_L'(\rho_L)
\qquad\text{strongly in }L^2(\Omega_T).
\]
Therefore, passing to the limit in
\[
q_{\delta,L}=\varepsilon^{-1}W_L'(\rho_{\delta,L})-\varepsilon\Delta\rho_{\delta,L}
\]
gives
\begin{equation}\label{eq:qL_weak_step4}
q_L=\varepsilon^{-1}W_L'(\rho_L)-\varepsilon\Delta\rho_L
\qquad\text{in }L^2(\Omega_T),
\end{equation}
equivalently, for a.e.\ \(t\in(0,T)\) and all \(\psi\in H^1(\Omega)\),
\[
\int_\Omega q_L(t)\psi\dx
=
\varepsilon^{-1}\int_\Omega W_L'(\rho_L(t))\psi\dx
+\varepsilon\int_\Omega \nabla\rho_L(t)\cdot\nabla\psi\dx.
\]

Define
\[
J_{\delta,L}^{\mathrm{diff}}:=m_{\delta,L}(\rho_{\delta,L})\nabla q_{\delta,L},
\qquad
J_{\delta,L}:=J_{\delta,L}^{\mathrm{diff}}+\kappa m_{\delta,L}(\rho_{\delta,L})D(\nu)\nabla\rho_{\delta,L}.
\]
From \eqref{eq:energy_deltaL} and \(m_{\delta,L}(\rho_{\delta,L})\le L\), we obtain
\[
\|J_{\delta,L}^{\mathrm{diff}}\|_{L^2(\Omega_T)}^2
\le
L\int_{\Omega_T} m_{\delta,L}(\rho_{\delta,L})|\nabla q_{\delta,L}|^2\dx\dt
\le C_L.
\]
Hence, after extraction,
\[
J_{\delta,L}^{\mathrm{diff}}\rightharpoonup J_L^{\mathrm{diff}}
\qquad\text{weakly in }L^2(\Omega_T;\mathbb{R}^d).
\]
Since \(m_{\delta,L}(\rho_{\delta,L})\to m_L(\rho_L)\) strongly in \(L^2(\Omega_T)\) and
\(\nabla\rho_{\delta,L}\to \nabla\rho_L\) strongly in \(L^2(\Omega_T)\), we also have
\[
m_{\delta,L}(\rho_{\delta,L})D(\nu)\nabla\rho_{\delta,L}
\to
m_L(\rho_L)D(\nu)\nabla\rho_L
\qquad\text{strongly in }L^1(\Omega_T;\mathbb{R}^d).
\]
We may therefore define
\[
J_L:=J_L^{\mathrm{diff}}+\kappa m_L(\rho_L)D(\nu)\nabla\rho_L
\in L^1(\Omega_T;\mathbb{R}^d).
\]

\begin{lemma}[Strong \(H^1\)-convergence of the truncated mobility for fixed \(L\)]
\label{lem:m_deltaL_to_mL_H1}
One has
\[
m_{\delta,L}(\rho_{\delta,L})\to m_L(\rho_L)
\qquad\text{strongly in }L^2(0,T;H^1(\Omega)).
\]
\end{lemma}

\begin{proof}
We first consider the \(L^2(\Omega_T)\)-part.
Since \(m_L(r)=\min\{r_+,L\}\) and \(m_{\delta,L}(r)=\min\{\max\{\delta,r\},L\}\), we have
\[
|m_{\delta,L}(r)-m_L(r)|\le \delta
\qquad\forall r\in\mathbb R.
\]
Therefore
\[
\begin{aligned}
\|m_{\delta,L}(\rho_{\delta,L})-m_L(\rho_L)\|_{L^2(\Omega_T)}
&\le
\|m_{\delta,L}(\rho_{\delta,L})-m_L(\rho_{\delta,L})\|_{L^2(\Omega_T)}
\\&\quad
+\|m_L(\rho_{\delta,L})-m_L(\rho_L)\|_{L^2(\Omega_T)}
\\
&\le
\delta |\Omega_T|^{1/2}+\|\rho_{\delta,L}-\rho_L\|_{L^2(\Omega_T)}
\to0,
\end{aligned}
\]
because \(m_L\) is \(1\)-Lipschitz.
For the gradients, the Sobolev chain rule gives
\[
\nabla m_{\delta,L}(\rho_{\delta,L})
=
\mathbf 1_{\{\delta<\rho_{\delta,L}<L\}}\nabla\rho_{\delta,L}
\qquad\text{a.e. in }\Omega_T,
\]
and
\[
\nabla m_L(\rho_L)
=
\mathbf 1_{\{0<\rho_L<L\}}\nabla\rho_L
\qquad\text{a.e. in }\Omega_T.
\]
Hence
\[
\begin{aligned}
&\|\nabla m_{\delta,L}(\rho_{\delta,L})-\nabla m_L(\rho_L)\|_{L^2(\Omega_T)}
\\
&\le
\|\mathbf 1_{\{\delta<\rho_{\delta,L}<L\}}(\nabla\rho_{\delta,L}-\nabla\rho_L)\|_{L^2(\Omega_T)}
\\&\quad
+\|(\mathbf 1_{\{\delta<\rho_{\delta,L}<L\}}-\mathbf 1_{\{0<\rho_L<L\}})\nabla\rho_L\|_{L^2(\Omega_T)}.
\end{aligned}
\]
The first term tends to zero by the strong convergence
\(\nabla\rho_{\delta,L}\to\nabla\rho_L\) in \(L^2(\Omega_T)\).
For the second term, note that
\[
\mathbf 1_{\{\delta<\rho_{\delta,L}<L\}}\to \mathbf 1_{\{0<\rho_L<L\}}
\qquad\text{a.e. on }\{\rho_L\notin\{0,L\}\}.
\]
Moreover, by the truncation/composition properties of Sobolev functions (see, e.g.,
\cite[Sec.~2.1]{ziemer1989weakly}), one has
\[
\nabla (\rho_L)_+ = \mathbf 1_{\{\rho_L>0\}}\nabla \rho_L,
\qquad
\nabla (\rho_L-L)_+ = \mathbf 1_{\{\rho_L>L\}}\nabla \rho_L
\quad\text{a.e. in }\Omega_T.
\]
In particular,
\[
\nabla\rho_L=0
\qquad\text{a.e. on }\{\rho_L=0\}\cup\{\rho_L=L\}.
\]
Therefore
\[
(\mathbf 1_{\{\delta<\rho_{\delta,L}<L\}}-\mathbf 1_{\{0<\rho_L<L\}})\nabla\rho_L
\to0
\qquad\text{a.e. in }\Omega_T.
\]
Since the absolute value is bounded by \(2|\nabla\rho_L|\in L^2(\Omega_T)\), dominated convergence yields
\[
\|(\mathbf 1_{\{\delta<\rho_{\delta,L}<L\}}-\mathbf 1_{\{0<\rho_L<L\}})\nabla\rho_L\|_{L^2(\Omega_T)}
\to0.
\]
This proves the claim.
\end{proof}

We next identify the diffusive flux at the \(L\)-level.
Let \(\zeta\in W^{1,\infty}(\Omega;\mathbb{R}^d)\) satisfy \(\zeta\cdot n=0\) on \(\partial\Omega\).
For every \(\delta\in(0,1)\), since
\(J_{\delta,L}^{\mathrm{diff}}=m_{\delta,L}(\rho_{\delta,L})\nabla q_{\delta,L}\),
integration by parts yields
\begin{equation}\label{eq:flux_id_deltaL}
\int_0^T\!\!\int_\Omega J_{\delta,L}^{\mathrm{diff}}\cdot \zeta\dx\dt
=
-\int_0^T\!\!\int_\Omega q_{\delta,L}\,\Div\!\bigl(m_{\delta,L}(\rho_{\delta,L})\zeta\bigr)\dx\dt.
\end{equation}
The boundary term vanishes because \(\zeta\cdot n=0\) on \(\partial\Omega\).
By Lemma~\ref{lem:m_deltaL_to_mL_H1},
\[
m_{\delta,L}(\rho_{\delta,L})\to m_L(\rho_L)
\qquad\text{strongly in }L^2(0,T;H^1(\Omega)),
\]
and therefore
\[
\Div\!\bigl(m_{\delta,L}(\rho_{\delta,L})\zeta\bigr)
\to
\Div\!\bigl(m_L(\rho_L)\zeta\bigr)
\qquad\text{strongly in }L^2(\Omega_T).
\]
Using also
\[
q_{\delta,L}\rightharpoonup q_L
\qquad\text{weakly in }L^2(\Omega_T),
\qquad
J_{\delta,L}^{\mathrm{diff}}\rightharpoonup J_L^{\mathrm{diff}}
\qquad\text{weakly in }L^2(\Omega_T;\mathbb{R}^d),
\]
we may pass to the limit in \eqref{eq:flux_id_deltaL} and obtain
\begin{equation}\label{eq:flux_id_L}
\int_0^T\!\!\int_\Omega J_L^{\mathrm{diff}}\cdot \zeta\dx\dt
=
-\int_0^T\!\!\int_\Omega q_L\,\Div\!\bigl(m_L(\rho_L)\zeta\bigr)\dx\dt.
\end{equation}

Passing to the limit in the weak formulation of \eqref{eq:regPDE_deltaL}, we obtain that for all
\(\varphi\in L^2(0,T;H^1(\Omega))\),
\begin{equation}\label{eq:weak_mass_L}
\int_0^T \langle \partial_t\rho_L,\varphi\rangle\dt
+\int_0^T\!\!\int_\Omega J_L\cdot\nabla\varphi\dx\dt
=0.
\end{equation}
Moreover, since
\[
\rho_{\delta,L}\to \rho_L
\qquad\text{strongly in } C([0,T];L^2(\Omega))
\]
and \(\rho_{\delta,L}(0)=\rho_0\) for every \(\delta\), we infer
\[
\rho_L(0)=\rho_0.
\]

To recover mass conservation, choose test functions of the form
\(\varphi(x,t)=\chi(t)\) with \(\chi\in C_c^\infty(0,T)\) in \eqref{eq:weak_mass_L}. Since
\(\nabla\varphi=0\), we obtain
\[
\int_0^T \left(\int_\Omega \rho_L(t)\dx\right)\chi'(t)\dt=0
\qquad\forall \chi\in C_c^\infty(0,T).
\]
Hence the map
\[
M_L(t):=\int_\Omega \rho_L(t)\dx
\]
is constant in \(\mathcal D'(0,T)\). Because \(\rho_L\in C([0,T];L^2(\Omega))\), the function
\(M_L\) is continuous on \([0,T]\), and therefore constant for all \(t\in[0,T]\). Using
\(\rho_L(0)=\rho_0\), we conclude that
\[
\int_\Omega \rho_L(t)\dx=\int_\Omega \rho_0\dx
\qquad\forall t\in[0,T].
\]

Set
\[
K_{\delta,L}:=\sqrt{m_{\delta,L}(\rho_{\delta,L})}\,\nabla q_{\delta,L}.
\]
Then \((K_{\delta,L})\) is bounded in \(L^2(\Omega_T;\mathbb{R}^d)\) by \eqref{eq:energy_deltaL}.
Since
\[
\sqrt{m_{\delta,L}(\rho_{\delta,L})}\to \sqrt{m_L(\rho_L)}
\qquad\text{strongly in }L^2(\Omega_T),
\]
after extraction we may write
\[
K_{\delta,L}\rightharpoonup K_L
\qquad\text{weakly in }L^2(\Omega_T;\mathbb{R}^d),
\]
and necessarily
\[
J_L^{\mathrm{diff}}=\sqrt{m_L(\rho_L)}\,K_L
\qquad\text{a.e. in }\Omega_T.
\]
Consequently,
\begin{equation}\label{eq:weightedJL_step4}
\int_{\Omega_T}\frac{|J_L^{\mathrm{diff}}|^2}{m_L(\rho_L)}\dx\dt
\le
\int_{\Omega_T}|K_L|^2\dx\dt
\le
\liminf_{\delta\downarrow0}\int_{\Omega_T}m_{\delta,L}(\rho_{\delta,L})|\nabla q_{\delta,L}|^2\dx\dt,
\end{equation}
where, by convention,
\[
\frac{|J_L^{\mathrm{diff}}|^2}{m_L(\rho_L)}:=0 \qquad \text{on }\{m_L(\rho_L)=0\}.
\]
The first inequality follows from \(J_L^{\mathrm{diff}}=\sqrt{m_L(\rho_L)}\,K_L\), which gives
\(|J_L^{\mathrm{diff}}|^2/m_L(\rho_L)=|K_L|^2\) on \(\{m_L(\rho_L)>0\}\) and \(0\) otherwise, so that the
weighted integral coincides with \(\int_{\{m_L(\rho_L)>0\}}|K_L|^2\le\int_{\Omega_T}|K_L|^2\).
Using also lower semicontinuity of the energy and the strong convergence
\[
m_{\delta,L}(\rho_{\delta,L})|\nabla\rho_{\delta,L}|^2
\to
m_L(\rho_L)|\nabla\rho_L|^2
\qquad\text{strongly in }L^1(\Omega_T),
\]
we may pass to the limit in \eqref{eq:energy_deltaL} and obtain
\begin{equation}\label{eq:energy_L}
\begin{aligned}
&E_L(\rho_L(t))
+\frac12\int_0^t\!\!\int_\Omega \frac{|J_L^{\mathrm{diff}}|^2}{m_L(\rho_L)}\dx\ds
\\
&\le
E_L(\rho_0)
+C(\Lambda)\kappa^2\int_0^t\!\!\int_\Omega m_L(\rho_L)|\nabla\rho_L|^2\dx\ds
\qquad\text{for a.e. }t\in(0,T).
\end{aligned}
\end{equation}

\subsubsection*{Step 5: Uniform (in \(L\)) entropy estimate and \(L^2(0,T;H^2)\)}

For \(L>1\), define the entropy density \(\Phi_L:[0,\infty)\to\mathbb{R}\) by
\begin{equation}\label{eq:PhiL_def}
\Phi_L(s):=
\begin{cases}
s\log s-s+1, & 0\le s\le L,\\[0.3em]
L\log L-L+1+\log(L)(s-L)+\dfrac{(s-L)^2}{2L}, & s\ge L.
\end{cases}
\end{equation}
Then \(\Phi_L\in C^1([0,\infty))\cap W^{2,\infty}_{\mathrm{loc}}([0,\infty))\), and
\[
\Phi_L''(s)=\frac1{m_L(s)} \qquad\text{for a.e. }s>0.
\]
Moreover, there exists \(C>0\), independent of \(L\), such that
\begin{equation}\label{eq:PhiL_quad_bound}
0\le \Phi_L(s)\le C(1+s^2)\qquad\forall s\ge0.
\end{equation}
Indeed, on \(0\le s\le L\), this follows uniformly in \(L\) from the elementary estimate
\(s\log s-s+1\le C(1+s^2)\): the function is bounded on \([0,1]\), while
\(s\log s\le s^2\) for \(s\ge1\). If \(s\ge L\), then
\[
\Phi_L(s)
=
L\log L-L+1+\log(L)(s-L)+\frac{(s-L)^2}{2L}.
\]
Using \(L\le s\) and \(\log L\le L\) for \(L>1\), we obtain
\[
L\log L-L+1\le L^2\le s^2,
\qquad
\log(L)(s-L)\le Ls\le s^2,
\]
and
\[
\frac{(s-L)^2}{2L}\le \frac{s^2}{2L}\le \frac{s^2}{2}.
\]
This proves the upper bound in \eqref{eq:PhiL_quad_bound}. The nonnegativity follows from convexity and the normalization
\(\Phi_L(1)=0\), \(\Phi_L'(1)=0\).

We emphasize that, at the degenerate \(L\)-level, we do \emph{not} test
the degenerate mass balance \eqref{eq:weak_mass_L} with \(\Phi_L'(\rho_L)\), since \(\Phi_L'\) is singular at \(0\).
Instead, we obtain the \(L\)-level entropy inequality by passing to the limit
\(\delta\downarrow0\) in \eqref{eq:entropy_deltaL}.
More precisely, for fixed \(L>1\), the functions \(\Phi_{\delta,L}\) converge locally uniformly on \([0,\infty)\) to \(\Phi_L\), and by Step~4 we have
\[
\rho_{\delta,L}\to \rho_L \qquad\text{strongly in }L^2(0,T;H^1(\Omega))\cap C([0,T];L^2(\Omega)).
\]
Moreover, the estimate on the negative part obtained in Step~3 yields
\[
\|(\rho_{\delta,L})_-\|_{L^\infty(0,T;L^2(\Omega))}\to0
\qquad\text{as }\delta\downarrow0.
\]
Since \(\rho_L\ge0\) a.e., it is enough to prove the pointwise lower-bound property
\begin{equation}\label{eq:Phi_deltaL_pointwise_liminf}
r_\delta\to r\ge0,\ \delta\downarrow0
\qquad\Longrightarrow\qquad
\liminf_{\delta\downarrow0}\Phi_{\delta,L}(r_\delta)\ge \Phi_L(r).
\end{equation}
Indeed, let \(r_\delta\to r\ge0\).
If \(r>0\), then for \(\delta\) small enough one has \(r_\delta\in [r/2,3r/2]\subset(0,\infty)\), and since
\(\Phi_{\delta,L}\to\Phi_L\) locally uniformly on \((0,\infty)\), it follows that
\[
\Phi_{\delta,L}(r_\delta)\to \Phi_L(r).
\]
If \(r=0\) and \(r_\delta\ge0\) along a subsequence, then the local uniform convergence on \([0,1]\) yields
\[
\Phi_{\delta,L}(r_\delta)\to \Phi_L(0)=1.
\]
If \(r=0\) and \(r_\delta<0\) along a subsequence, then, since
\(\Phi_{\delta,L}''\ge0\) and \(\Phi_{\delta,L}'(1)=0\), we have
\(\Phi_{\delta,L}'(s)\le0\) for all \(s\le1\), hence \(\Phi_{\delta,L}\) is decreasing on \((-\infty,1]\).
Therefore
\[
\Phi_{\delta,L}(r_\delta)\ge \Phi_{\delta,L}(0).
\]
Moreover, a direct computation from the definition of \(\Phi_{\delta,L}\) gives
\[
\Phi_{\delta,L}(0)=1-\frac{\delta}{2}\longrightarrow 1=\Phi_L(0).
\]
Thus \eqref{eq:Phi_deltaL_pointwise_liminf} follows.

Now fix \(t\in[0,T]\) such that
\[
\rho_{\delta,L}(t,x)\to \rho_L(t,x)\qquad\text{for a.e. }x\in\Omega,
\]
which holds for a.e.\ \(t\in(0,T)\) by the a.e.\ convergence in \(\Omega_T\).
Applying \eqref{eq:Phi_deltaL_pointwise_liminf} pointwise with
\(r_\delta=\rho_{\delta,L}(t,x)\) and using Fatou's lemma, we obtain
\[
\int_\Omega \Phi_L(\rho_L(t))\dx
\le
\liminf_{\delta\downarrow0}\int_\Omega \Phi_{\delta,L}(\rho_{\delta,L}(t))\dx .
\]
Since \(\rho_0\ge0\) and \(\Phi_{\delta,L}\to\Phi_L\) locally uniformly on \([0,\infty)\), while for fixed \(L\)
the family \((\Phi_{\delta,L})_{\delta\in(0,1)}\) has at most quadratic growth on \([0,\infty)\),
dominated convergence yields
\[
\int_\Omega \Phi_{\delta,L}(\rho_0)\dx \to \int_\Omega \Phi_L(\rho_0)\dx .
\]
Using also weak lower semicontinuity and the strong convergence
\(\nabla\rho_{\delta,L}\to \nabla\rho_L\) in \(L^2(\Omega_T)\), we may pass to the limit in
\eqref{eq:entropy_deltaL} and obtain that for a.e.\ \(t\in(0,T)\),
\begin{equation}\label{eq:entropy_L_diff}
\begin{aligned}
&\int_\Omega \Phi_L(\rho_L(t)) \dx
+\varepsilon\int_0^t \int_\Omega |\Delta\rho_L|^2 \dx\ds
+\kappa\int_0^t \int_\Omega \nabla\rho_L\cdot D(\nu)\nabla\rho_L \dx\ds
\\
&\le
\int_\Omega \Phi_L(\rho_0) \dx
+\frac{c_1}{\varepsilon}\int_0^t \int_\Omega |\nabla\rho_L|^2 \dx\ds .
\end{aligned}
\end{equation}

To close this estimate uniformly in \(L\), we use mass conservation together with the Neumann elliptic estimate.
Let
\[
\bar\rho:=|\Omega|^{-1}\int_\Omega \rho_0 \dx, \qquad \tilde\rho_L:=\rho_L-\bar\rho.
\]
Then
\[
\int_\Omega \tilde\rho_L(t)\dx=0
\qquad\text{for a.e. }t\in(0,T),
\]
and, by the weak formulation of the chemical-potential relation together with
\(\rho_L\in L^2(0,T;H^2(\Omega))\), one has
\[
\partial_n \tilde\rho_L=\partial_n \rho_L=0
\qquad\text{on }\partial\Omega
\]
in the trace sense for a.e. \(t\in(0,T)\).
Therefore the standard \(H^2\)-regularity estimate for the Neumann Laplacian on \(C^{1,1}\) domains yields
\[
\|\tilde\rho_L(t)\|_{H^2(\Omega)}
\le C\|\Delta \tilde\rho_L(t)\|_{L^2(\Omega)}
= C\|\Delta \rho_L(t)\|_{L^2(\Omega)}
\qquad\text{for a.e. }t\in(0,T).
\]
Moreover, since \(\bar\rho\) is constant in space,
\[
\|\bar\rho\|_{H^2(\Omega)}=\|\bar\rho\|_{L^2(\Omega)}
\le C\|\rho_0\|_{L^1(\Omega)}.
\]
Hence
\begin{equation}\label{eq:H2_by_Delta}
\|\rho_L(t)\|_{H^2(\Omega)}
\le C\|\Delta\rho_L(t)\|_{L^2(\Omega)}+C\|\rho_0\|_{L^1(\Omega)}
\qquad\text{for a.e. }t\in(0,T).
\end{equation}

By Ehrling's lemma \cite[Lemma II.5.15]{boyer2012mathematical},
for every \(\tau>0\) there exists \(C_\tau>0\) such that
\[
\|\nabla\rho_L\|_{L^2(\Omega)}
\le \tau\|\rho_L\|_{H^2(\Omega)}+C_\tau\|\rho_L\|_{L^1(\Omega)}.
\]
Using \eqref{eq:H2_by_Delta} together with mass conservation
\(\|\rho_L(t)\|_{L^1(\Omega)}=\|\rho_0\|_{L^1(\Omega)}\) (which follows from \(\rho_L\ge0\)
a.e.\ and the conservation of \(\int_\Omega \rho_L\dx\) established in Step~4),
we obtain
\[
\|\nabla\rho_L\|_{L^2(\Omega)}^2
\le
2\tau^2C^2\|\Delta\rho_L\|_{L^2(\Omega)}^2
+2\bigl(\tau C+C_\tau\bigr)^2\|\rho_0\|_{L^1(\Omega)}^2.
\]
Insert this estimate into \eqref{eq:entropy_L_diff} and choose \(\tau>0\) so small that
\[
\frac{c_1}{\varepsilon}\cdot 2\tau^2C^2\le \frac{\varepsilon}{2}.
\]
Then there exists \(C_0>0\), independent of \(L\), such that
\begin{equation}\label{eq:entropy_L_uniform}
\int_\Omega \Phi_L(\rho_L(t)) \dx
+\frac{\varepsilon}{2}\int_0^t \int_\Omega |\Delta\rho_L|^2 \dx\ds
\le
\int_\Omega \Phi_L(\rho_0) \dx + C_0 t
\qquad\text{for a.e. } t\in(0,T).
\end{equation}
Since \(\Phi_L\) has at most quadratic growth uniformly in \(L\), we have
\[
\sup_{L>1}\int_\Omega \Phi_L(\rho_0)\dx<\infty.
\]
Therefore the right-hand side of \eqref{eq:entropy_L_uniform}
is bounded uniformly in \(L\), and in particular
\begin{equation}\label{eq:Delta_bound_uniform}
\int_0^T  \int_\Omega |\Delta\rho_L|^2 \dx dt \le C,
\qquad \sup_{t\in[0,T]}\int_\Omega \Phi_L(\rho_L(t)) \dx\le C,
\end{equation}
with \(C\) independent of \(L\).

\subsubsection*{Step 6: Uniform (in \(L\)) energy estimate}

From \eqref{eq:energy_L} and \(m_L(\rho_L)\le \rho_L\) (since \(\rho_L\ge 0\)) we have
\begin{equation}\label{eq:energy_L_rhs}
E_L(\rho_L(t))
\le E_L(\rho_0)
+C(\Lambda)\kappa^2\int_0^t  \int_\Omega \rho_L |\nabla\rho_L|^2 \dx\ds.
\end{equation}
Regarding the right-hand side, we estimate
\[
\int_\Omega \rho_L |\nabla\rho_L|^2 \dx \le \|\rho_L\|_{L^\infty(\Omega)}\|\nabla\rho_L\|_{L^2(\Omega)}^2.
\]
Since \(d\le 3\), the embedding \(H^2(\Omega)\hookrightarrow L^\infty(\Omega)\) holds, and by \eqref{eq:H2_by_Delta}
\[
\|\rho_L\|_{L^\infty(\Omega)}\le C\|\rho_L\|_{H^2(\Omega)}
\le C\|\Delta\rho_L\|_{L^2(\Omega)}+C\|\rho_0\|_{L^1(\Omega)}.
\]
Moreover, by the definition of \(E_L\),
\[
\|\nabla\rho_L\|_{L^2(\Omega)}^2\le \frac{2}{\varepsilon}E_L(\rho_L)+C.
\]
Thus there exists \(C>0\), independent of \(L\), such that
\[
\int_\Omega \rho_L |\nabla\rho_L|^2 \dx
\le
C\bigl(1+\|\Delta\rho_L\|_{L^2(\Omega)}\bigr)\bigl(1+E_L(\rho_L)\bigr).
\]
Integrating in time and using Cauchy--Schwarz together with \eqref{eq:Delta_bound_uniform} gives
\[
\int_0^t  \int_\Omega \rho_L |\nabla\rho_L|^2 \dx\ds
\le
C
+ C\int_0^t \|\Delta\rho_L(s)\|_{L^2(\Omega)} E_L(\rho_L(s)) \ds
+ C\int_0^t E_L(\rho_L(s)) \ds.
\]
Since \(\|\Delta\rho_L\|_{L^2(0,T;L^2)}\le C\), the function
\(g_L(s):=\|\Delta\rho_L(s)\|_{L^2(\Omega)}\) belongs to \(L^1(0,T)\) uniformly in \(L\).
Therefore, \eqref{eq:energy_L_rhs} implies an integral Grönwall inequality of the form
\[
E_L(\rho_L(t)) \le C + C\int_0^t \bigl(1+g_L(s)\bigr)E_L(\rho_L(s)) ds,
\]
with \(\int_0^T (1+g_L(s)) ds\le C\) uniformly in \(L\).
By Grönwall's lemma,
\begin{equation}\label{eq:energy_L_uniform}
\sup_{t\in[0,T]} E_L(\rho_L(t))\le C.
\end{equation}
Combining \eqref{eq:energy_L_uniform} with \eqref{eq:Delta_bound_uniform} yields
\begin{equation}\label{eq:uniform_L_bounds}
\|\rho_L\|_{L^\infty(0,T;H^1(\Omega))}
+\|\rho_L\|_{L^2(0,T;H^2(\Omega))}
\le C,
\end{equation}
independent of \(L\).

We next derive an improved bound on the time derivative.
Set
\[
J_L:=J_L^{\mathrm{diff}}+\kappa m_L(\rho_L)D(\nu)\nabla\rho_L.
\]
We first improve the integrability of the diffusive flux.
By \eqref{eq:energy_L},
\[
\int_{\Omega_T}\frac{|J_L^{\mathrm{diff}}|^2}{m_L(\rho_L)} \dx\dt \le C.
\]
Moreover, since
\[
\rho_L\in L^\infty(0,T;L^6(\Omega))
\cap L^2(0,T;L^\infty(\Omega)),
\]
Lebesgue interpolation in time yields
\[
\rho_L\in L^8(\Omega_T),
\qquad
\|\rho_L\|_{L^8(\Omega_T)}
\le C
\|\rho_L\|_{L^\infty(0,T;L^6(\Omega))}^{3/4}
\|\rho_L\|_{L^2(0,T;L^\infty(\Omega))}^{1/4}.
\]
Since \(0\le m_L(\rho_L)\le \rho_L\), we also have
\[
\|m_L(\rho_L)\|_{L^8(\Omega_T)}\le \|\rho_L\|_{L^8(\Omega_T)}\le C.
\]
Now observe that
\[
|J_L^{\mathrm{diff}}|^{16/9}
=
\left(\frac{|J_L^{\mathrm{diff}}|^2}{m_L(\rho_L)}\right)^{8/9}
m_L(\rho_L)^{8/9}.
\]
Hence, by Hölder's inequality with exponents \(9/8\) and \(9\),
\[
\int_{\Omega_T}|J_L^{\mathrm{diff}}|^{16/9}\dx\dt
\le
\left(\int_{\Omega_T}\frac{|J_L^{\mathrm{diff}}|^2}{m_L(\rho_L)}\dx\dt\right)^{8/9}
\left(\int_{\Omega_T}m_L(\rho_L)^8\dx\dt\right)^{1/9}.
\]
Therefore,
\[
\|J_L^{\mathrm{diff}}\|_{L^{16/9}(\Omega_T)}\le C.
\]

We next estimate the traction contribution.
Since \(m_L(\rho_L)\le \rho_L\) and \(D(\nu)\in L^\infty(\Omega_T)\), it is enough to control
\(\rho_L\nabla\rho_L\) in \(L^2(\Omega_T)\).
Consequently,
\[
\|m_L(\rho_L)D(\nu)\nabla\rho_L\|_{L^2(\Omega_T)}
\le
C\|\rho_L\|_{L^\infty(0,T;L^4(\Omega))}
\|\nabla\rho_L\|_{L^2(0,T;L^4(\Omega))}
\le C.
\]
Since \(L^2(\Omega_T)\hookrightarrow L^{16/9}(\Omega_T)\) on the finite-measure cylinder \(\Omega_T\), we conclude that
\[
\|J_L\|_{L^{16/9}(\Omega_T)}\le C.
\]

Now let \(\varphi\in L^{16/7}(0,T;W^{1,16/7}(\Omega))\). By the weak mass balance,
\[
\int_0^T \langle \partial_t\rho_L,\varphi\rangle\dt
=
-\int_0^T\!\!\int_\Omega J_L\cdot \nabla\varphi \dx\dt.
\]
Hence, by Hölder's inequality,
\[
\left|\int_0^T \langle \partial_t\rho_L,\varphi\rangle\dt\right|
\le
\|J_L\|_{L^{16/9}(\Omega_T)}
\|\nabla\varphi\|_{L^{16/7}(\Omega_T)}
\le
C\|\varphi\|_{L^{16/7}(0,T;W^{1,16/7}(\Omega))}.
\]
Therefore,
\begin{equation}\label{eq:dt_bound_improved}
\|\partial_t\rho_L\|_{L^{16/9}(0,T;(W^{1,16/7}(\Omega))')} \le C.
\end{equation}

\subsubsection*{Step 7: Passage \(L\to\infty\)}
By \eqref{eq:uniform_L_bounds} and the improved time-derivative bound \eqref{eq:dt_bound_improved},
Simon's compactness theorem \cite[Corollary~4]{simon1986compact}
yields the existence of \(\rho\) and a subsequence (not relabeled) such that
\[
\begin{aligned}
\rho_L &\rightharpoonup^\ast \rho &&\text{in }L^\infty(0,T;H^1(\Omega)), \\
\rho_L &\rightharpoonup \rho &&\text{in }L^2(0,T;H^2(\Omega)), \\
\rho_L &\to \rho &&\text{strongly in }L^2(0,T;H^1(\Omega)),\\
\rho_L &\to \rho &&\text{strongly in }C([0,T];L^2(\Omega)),\\
\rho_L &\to \rho &&\text{strongly in }L^2(0,T;L^\infty(\Omega))
\cap L^2(0,T;W^{1,4}(\Omega)).
\end{aligned}
\]
In particular, after extraction of a further subsequence,
\[
\rho_L(x,t)\to \rho(x,t)\qquad\text{for a.e. }(x,t)\in\Omega_T,
\]
and \(\rho\ge0\) a.e. Moreover,
\begin{equation}\label{eq:strong_comp_extra}
\rho_L\to \rho \quad\text{strongly in }L^2(0,T;L^\infty(\Omega))
\cap L^2(0,T;W^{1,4}(\Omega)).
\end{equation}
In particular,
\[
\nabla\rho_L\to \nabla\rho
\qquad\text{strongly in }L^2(0,T;L^4(\Omega)).
\]
Since \(m_L(\rho_L)=\min\{\rho_L,L\}\to \rho\) a.e., we next show strong convergence in \(L^2(\Omega_T)\).
By Chebyshev and the uniform \(L^p(\Omega_T)\) bound (from \eqref{eq:energy_L_uniform} together with the
\(L\)-independent coercivity \(W_L(s)\ge c_2|s|^p-c_3\) of Lemma~\ref{lem:WL}(4)),
\[
|\{\rho_L\ge L\}|\le L^{-p}\int_{\Omega_T}\rho_L^p \dx dt \le C L^{-p}\to 0.
\]
Thus
\[
\begin{aligned}
\|m_L(\rho_L)-\rho_L\|_{L^2(\Omega_T)}^2
&\le \int_{\{\rho_L\ge L\}} |\rho_L-L|^2 \dx \dt
\\ &\le \int_{\{\rho_L\ge L\}} \rho_L^2 \dx \dt
\\ &\le \|\rho_L\|_{L^p(\Omega_T)}^2 |\{\rho_L\ge L\}|^{1-\frac{2}{p}}\to 0,
\end{aligned}
\]
and since \(\rho_L\to \rho\) in \(L^2(\Omega_T)\) we conclude
\begin{equation}\label{eq:mL_to_rho}
m_L(\rho_L)\to \rho \qquad \text{strongly in }L^2(\Omega_T).
\end{equation}

We in fact have the stronger convergence
\begin{equation}\label{eq:mL_to_rho_H1}
m_L(\rho_L)\to \rho
\qquad\text{strongly in }L^2(0,T;H^1(\Omega)).
\end{equation}
Indeed, since \(\rho_L\ge0\) a.e., the chain rule yields
\[
\nabla m_L(\rho_L)=\mathbf{1}_{\{0<\rho_L<L\}}\nabla\rho_L
\qquad\text{a.e. in }\Omega_T,
\]
and therefore
\[
\nabla m_L(\rho_L)-\nabla\rho_L
=
-\mathbf{1}_{\{\rho_L\ge L\}}\nabla\rho_L.
\]
Let \(A_L:=\{(x,t)\in\Omega_T:\rho_L(x,t)\ge L\}\). By interpolation between
\(L^\infty(0,T;L^2(\Omega))\) and \(L^2(0,T;L^6(\Omega))\) (which both contain
\(\nabla\rho_L\) by \eqref{eq:uniform_L_bounds} and the Sobolev embedding for \(d\le 3\)),
we have
\[
\nabla\rho_L \in L^{10/3}(\Omega_T;\mathbb{R}^d)\qquad\text{uniformly in }L,
\]
and consequently \(|\nabla\rho_L|^2\) is bounded in \(L^{5/3}(\Omega_T)\) uniformly in \(L\),
hence equi-integrable. Since \(|A_L|\to0\) by the Chebyshev estimate above,
the absolute continuity of the integral yields
\[
\|\nabla m_L(\rho_L)-\nabla\rho_L\|_{L^2(\Omega_T)}^2
=\int_{A_L} |\nabla\rho_L|^2 \dx\dt \to 0.
\]
Since \(\rho_L\to\rho\) strongly in \(L^2(0,T;H^1(\Omega))\), this proves \eqref{eq:mL_to_rho_H1}.

Define \(J_L:=J_L^{\mathrm{diff}}+\kappa m_L(\rho_L)D(\nu)\nabla\rho_L\).
From \eqref{eq:energy_L} we have
\[
\int_{\Omega_T}\frac{|J_L^{\mathrm{diff}}|^2}{m_L(\rho_L)} \dx\dt\le C.
\]
Set
\[
K_L:=
\begin{cases}
\dfrac{J_L^{\mathrm{diff}}}{\sqrt{m_L(\rho_L)}}, & \text{if } m_L(\rho_L)>0,\\[0.8em]
0, & \text{if } m_L(\rho_L)=0.
\end{cases}
\]
Then \((K_L)\) is bounded in \(L^2(\Omega_T;\mathbb{R}^d)\), so, up to a subsequence,
\[
K_L\rightharpoonup K
\qquad\text{weakly in }L^2(\Omega_T;\mathbb{R}^d).
\]
Since \(a,b\ge0\) implies \(|\sqrt a-\sqrt b|^2\le |a-b|\), the strong convergence
\eqref{eq:mL_to_rho} yields
\[
\sqrt{m_L(\rho_L)}\to \sqrt{\rho}
\qquad\text{strongly in }L^2(\Omega_T).
\]
Hence
\[
J_L^{\mathrm{diff}}
=
\sqrt{m_L(\rho_L)}\,K_L
\rightharpoonup
\sqrt{\rho}\,K=:J^{\mathrm{diff}}
\qquad\text{weakly in }L^1(\Omega_T;\mathbb{R}^d).
\]

For the traction part, \eqref{eq:mL_to_rho_H1} and
\(\nabla\rho_L\rightharpoonup \nabla\rho\) in \(L^2(\Omega_T)\) imply
\[
m_L(\rho_L)D(\nu)\nabla\rho_L
\rightharpoonup
\rho D(\nu)\nabla\rho
\qquad\text{weakly in }L^1(\Omega_T;\mathbb{R}^d).
\]
Therefore
\[
J_L \rightharpoonup J:=J^{\mathrm{diff}}+\kappa\rho D(\nu)\nabla\rho
\qquad\text{weakly in }L^1(\Omega_T;\mathbb{R}^d),
\]
which gives \eqref{eq:Jdiff_def}.

To pass to the constitutive identity, let
\(\zeta\in W^{1,\infty}(\Omega;\mathbb{R}^d)\) satisfy \(\zeta\cdot n=0\) on \(\partial\Omega\).
By \eqref{eq:flux_id_L},
\[
\int_0^T\!\!\int_\Omega J_L^{\mathrm{diff}}\cdot \zeta \dx\dt
=
-\int_0^T\!\!\int_\Omega q_L\,\Div\!\bigl(m_L(\rho_L)\zeta\bigr)\dx\dt.
\]
By \eqref{eq:mL_to_rho_H1} we have
\[
m_L(\rho_L)\to \rho
\qquad\text{strongly in }L^2(0,T;H^1(\Omega)),
\]
and therefore
\[
\Div\bigl(m_L(\rho_L)\zeta\bigr)\to \Div(\rho\zeta)
\qquad\text{strongly in }L^2(\Omega_T).
\]
Since \(q_L\rightharpoonup q\) weakly in \(L^2(\Omega_T)\), we may pass to the limit and obtain
\[
\int_0^T\!\!\int_\Omega J^{\mathrm{diff}}\cdot \zeta \dx\dt
=
-\int_0^T\!\!\int_\Omega q\, \Div(\rho\zeta)\dx\dt,
\]
which is exactly \eqref{eq:flux_id}.

Finally, since \(J^{\mathrm{diff}}=\sqrt{\rho}\,K\) gives
\(|J^{\mathrm{diff}}|^2/\rho=|K|^2\) on \(\{\rho>0\}\) (and \(0\) on \(\{\rho=0\}\) by convention),
weak lower semicontinuity yields \eqref{eq:weightedJdiff} due to
\[
\begin{aligned}
\int_{\Omega_T}\frac{|J^{\mathrm{diff}}|^2}{\rho}\dx\dt
=
\int_{\{\rho>0\}}|K|^2\dx\dt
\le
\int_{\Omega_T}|K|^2\dx\dt
&\le
\liminf_{L\to\infty}\int_{\Omega_T}|K_L|^2\dx\dt
\\ &=
\liminf_{L\to\infty}\int_{\Omega_T} \frac{|J_L^{\mathrm{diff}}|^2}{m_L(\rho_L)} \dx\dt.
\end{aligned}
\]

By
\(
q_L=\varepsilon^{-1}W_L'(\rho_L)-\varepsilon\Delta\rho_L
\)
and the uniform \(L^2(\Omega_T)\)-bound on \(\Delta\rho_L\), it remains to control
\(W_L'(\rho_L)\) in \(L^2(\Omega_T)\), uniformly in \(L\).
Using Lemma~\ref{lem:WL}(3) and the embedding
\(H^1(\Omega)\hookrightarrow L^6(\Omega)\) for \(d\le3\), we obtain
\[
\|W_L'(\rho_L)\|_{L^2(\Omega_T)}
\le
C\Bigl(1+\|\rho_L\|_{L^{2(p-1)}(\Omega_T)}^{p-1}\Bigr)\le C.
\]
Hence \((q_L)\) is bounded in \(L^2(\Omega_T)\), and after extraction
\[
q_L\rightharpoonup q
\qquad\text{weakly in }L^2(\Omega_T).
\]

By \eqref{eq:uniform_L_bounds} and \eqref{eq:dt_bound_improved}, Simon's compactness theorem also yields
\[
\rho_L\to \rho
\qquad\text{strongly in }L^2(0,T;L^{2(p-1)}(\Omega)),
\]
because \(H^2(\Omega)\hookrightarrow\hookrightarrow L^{2(p-1)}(\Omega)\) for \(d\le3\).
Therefore, by continuity of the Nemytskii map associated with \(W'\),
\[
W'(\rho_L)\to W'(\rho)
\qquad\text{strongly in }L^2(\Omega_T).
\]
Moreover, if we set
\[
A_L:=\{(x,t)\in\Omega_T:\rho_L(x,t)\ge L\},
\]
then
\[
\|(W_L'-W')(\rho_L)\|_{L^2(\Omega_T)}^2
=
\int_{A_L} |W_L'(\rho_L)-W'(\rho_L)|^2 \dx\dt.
\]
By Lemma~\ref{lem:WL}(3),
\[
|W_L'(s)-W'(s)|^2\le C\bigl(1+|s|^{2p-2}\bigr)
\qquad\forall s\in\mathbb{R},
\]
hence
\[
\|(W_L'-W')(\rho_L)\|_{L^2(\Omega_T)}^2
\le
C\int_{A_L}\bigl(1+\rho_L^{2p-2}\bigr)\dx\dt.
\]
Since \(d\le 3\) and \(p\le 4\), we have \(2p-2\le 6<8\), so the uniform bound
\(\sup_{L>1}\|\rho_L\|_{L^8(\Omega_T)}<\infty\) from Step~6 gives, by Hölder's inequality,
\[
\int_{A_L}\bigl(1+\rho_L^{2p-2}\bigr)\dx\dt
\le
\bigl(1+\|\rho_L\|_{L^8(\Omega_T)}^{2p-2}\bigr)\,|A_L|^{1-\frac{2p-2}{8}}.
\]
Since \(|A_L|\to0\) and the exponent \(1-\tfrac{2p-2}{8}>0\), the right-hand side tends to \(0\).
Therefore
\[
\|(W_L'-W')(\rho_L)\|_{L^2(\Omega_T)}\to 0.
\]
Consequently,
\[
W_L'(\rho_L)\to W'(\rho)
\qquad\text{strongly in }L^2(\Omega_T).
\]

To identify the limit equation for \(q\), let \(\chi\in C_c^\infty(0,T)\) and \(\psi\in H^1(\Omega)\).
For every \(L\),
\[
\int_0^T \chi(t)\int_\Omega q_L(t)\psi \dx\dt
=
\varepsilon^{-1}\int_0^T \chi(t)\int_\Omega W_L'(\rho_L(t))\psi \dx\dt
+\varepsilon\int_0^T \chi(t)\int_\Omega \nabla\rho_L(t)\cdot\nabla\psi \dx\dt.
\]
Using the weak convergence \(q_L\rightharpoonup q\) in \(L^2(\Omega_T)\), the strong convergence
\(W_L'(\rho_L)\to W'(\rho)\) in \(L^2(\Omega_T)\), and
\(\nabla\rho_L\rightharpoonup \nabla\rho\) in \(L^2(\Omega_T)\), we may pass to the limit and obtain
\[
\int_0^T \chi(t)\int_\Omega q(t)\psi \dx\dt
=
\varepsilon^{-1}\int_0^T \chi(t)\int_\Omega W'(\rho(t))\psi \dx\dt
+\varepsilon\int_0^T \chi(t)\int_\Omega \nabla\rho(t)\cdot\nabla\psi \dx\dt.
\]
Since \(\chi\in C_c^\infty(0,T)\) is arbitrary, \eqref{eq:q_weak} follows for a.e. \(t\in(0,T)\).

By construction, \(\Phi_L(s)=\Phi(s)\) for \(0\le s\le L\), and for every fixed \(s\ge0\) one has
\(
\Phi_L(s)=\Phi(s)\) for all \(L\ge s.
\)
In particular,
\[
\Phi_L(s)\to \Phi(s)\qquad\text{for every }s\ge0.
\]
Moreover, \(\Phi_L(s)\ge \Phi(s)\) for all \(s\ge0\), since \(\Phi_L=\Phi\) on \([0,L]\) and
\(\Phi_L''(s)=1/L\ge 1/s=\Phi''(s)\) for \(s\ge L\), with matching value and first derivative at \(s=L\).
Finally, by \eqref{eq:PhiL_quad_bound}, there exists \(C>0\), independent of \(L\), such that
\[
0\le \Phi_L(s)\le C(1+s^2)\qquad\forall s\ge0.
\]
Since \(\Phi_L(\rho_0)\to \Phi(\rho_0)\) a.e.\ in \(\Omega\) and
\(0\le \Phi_L(\rho_0)\le C(1+\rho_0^2)\in L^1(\Omega)\), dominated convergence yields
\begin{equation}\label{eq:PhiL_rho0_conv}
\int_\Omega \Phi_L(\rho_0)\dx \to \int_\Omega \Phi(\rho_0)\dx.
\end{equation}

Fix \(t\in[0,T]\).
Since \(\rho_L(t)\to \rho(t)\) strongly in \(L^2(\Omega)\), every subsequence of \((\rho_L(t))_{L>1}\)
contains a further subsequence (not relabeled) such that
\[
\rho_L(t,x)\to \rho(t,x)\qquad\text{for a.e. }x\in\Omega.
\]
Because \(\rho_L\ge0\) a.e., we claim that
\begin{equation}\label{eq:PhiL_pointwise_liminf}
r_L\to r\ge0,\ L\to\infty
\qquad\Longrightarrow\qquad
\liminf_{L\to\infty}\Phi_L(r_L)\ge \Phi(r).
\end{equation}
Indeed, if \(r>0\), then for \(L\) sufficiently large one has \(r_L\in [r/2,3r/2]\subset (0,L)\), hence
\[
\Phi_L(r_L)=\Phi(r_L)\to \Phi(r).
\]
If \(r=0\), then \(r_L\to0\) and for \(L\) large enough again \(r_L<L\), so that
\[
\Phi_L(r_L)=\Phi(r_L)\to \Phi(0)=1.
\]
This proves \eqref{eq:PhiL_pointwise_liminf}.

Applying \eqref{eq:PhiL_pointwise_liminf} pointwise with \(r_L=\rho_L(t,x)\) and using Fatou's lemma, we obtain
\[
\int_\Omega \Phi(\rho(t))\dx
\le
\liminf_{L\to\infty}\int_\Omega \Phi_L(\rho_L(t))\dx .
\]
Moreover, by weak lower semicontinuity,
\[
\int_0^t\!\!\int_\Omega |\Delta\rho|^2\dx\ds
\le
\liminf_{L\to\infty}\int_0^t\!\!\int_\Omega |\Delta\rho_L|^2\dx\ds.
\]
Furthermore, since \(D(\nu)\) is symmetric and nonnegative definite a.e., we may define its measurable square root
\(D(\nu)^{1/2}\in L^\infty(\Omega_T;\mathbb{R}^{d\times d})\). As
\[
D(\nu)^{1/2}\nabla\rho_L \rightharpoonup D(\nu)^{1/2}\nabla\rho
\qquad\text{weakly in }L^2(\Omega_T;\mathbb{R}^d),
\]
another application of weak lower semicontinuity gives
\[
\int_0^t\!\!\int_\Omega \nabla\rho\cdot D(\nu)\nabla\rho \dx\ds
\le
\liminf_{L\to\infty}\int_0^t\!\!\int_\Omega \nabla\rho_L\cdot D(\nu)\nabla\rho_L \dx\ds.
\]
Passing to the limit in \eqref{eq:entropy_L_diff} and using \eqref{eq:PhiL_rho0_conv} therefore yields the entropy inequality \eqref{eq:entropyineq}.

Since \(\rho_L\to\rho\) a.e.\ and \(W_L=W\) on \([-L,L]\), we have
\[
W_L(\rho_L)\to W(\rho)\qquad\text{a.e. in }\Omega_T.
\]
Moreover, by Lemma~\ref{lem:WL}(4) and \(p>2\), there exists \(C>0\), independent of \(L\), such that
\[
W_L(s)\ge -C(1+s^2)\qquad\forall s\in\mathbb{R}.
\]
Using Fatou's lemma on \(W_L(\rho_L)+C(1+\rho_L^2)\ge0\), together with the pointwise a.e.\ convergence and the strong convergence \(\rho_L(t)\to\rho(t)\) in \(L^2(\Omega)\) for a.e.\ \(t\), we obtain
\[
\int_\Omega W(\rho(t))\dx \le \liminf_{L\to\infty}\int_\Omega W_L(\rho_L(t))\dx .
\]
Together with weak lower semicontinuity of \(\|\nabla\rho(t)\|_{L^2(\Omega)}^2\), this gives
\[
E(\rho(t))\le \liminf_{L\to\infty} E_L(\rho_L(t)).
\]

Next, by the flux convergence argument and weak lower semicontinuity,
\[
\int_0^t\!\!\int_\Omega \frac{|J^{\mathrm{diff}}|^2}{\rho}\dx\ds
\le
\liminf_{L\to\infty}\int_0^t\!\!\int_\Omega \frac{|J_L^{\mathrm{diff}}|^2}{m_L(\rho_L)}\dx\ds.
\]

It remains to pass to the right-hand side involving the traction term.
We use Vitali's convergence theorem.
The almost everywhere convergence \(\rho_L\to\rho\) and \(\nabla\rho_L\to\nabla\rho\)
(the latter after extraction of a further subsequence, using
\eqref{eq:strong_comp_extra}) gives
\[
m_L(\rho_L)|\nabla\rho_L|^2 \to \rho|\nabla\rho|^2
\qquad\text{a.e. in }\Omega_T.
\]
For equi-integrability, note that by Step 6 we have \(\rho_L\in L^8(\Omega_T)\) uniformly in \(L\),
and by the interpolation argument used above, \(|\nabla\rho_L|^2\in L^{5/3}(\Omega_T)\) uniformly in \(L\).
Since \(0\le m_L(\rho_L)\le \rho_L\), Hölder's inequality applied to the product
\(m_L(\rho_L)\cdot|\nabla\rho_L|^2\) with \(\rho_L\in L^8(\Omega_T)\) and
\(|\nabla\rho_L|^2\in L^{5/3}(\Omega_T)\) gives
\[
\bigl\| m_L(\rho_L)|\nabla\rho_L|^2 \bigr\|_{L^{r}(\Omega_T)} \le C,
\qquad
\frac1r=\frac18+\frac35=\frac{29}{40},
\quad\text{i.e. } r=\frac{40}{29}>1,
\]
uniformly in \(L\). In particular, \((m_L(\rho_L)|\nabla\rho_L|^2)_{L>1}\) is
equi-integrable on \(\Omega_T\). By Vitali's theorem,
\[
\int_0^t\!\!\int_\Omega m_L(\rho_L)|\nabla\rho_L|^2\dx\ds
\to
\int_0^t\!\!\int_\Omega \rho|\nabla\rho|^2\dx\ds.
\]

Passing to the limit in \eqref{eq:energy_L} now yields \eqref{eq:energyineq}.

\smallskip
\noindent\textit{Initial condition.}
Since \(\rho_L\to \rho\) strongly in \(C([0,T];L^2(\Omega))\) and \(\rho_L(0)=\rho_0\) for every \(L\),
evaluation at \(t=0\) gives
\[
\|\rho(0)-\rho_0\|_{L^2(\Omega)}
=
\|\rho(0)-\rho_L(0)\|_{L^2(\Omega)}
\le
\sup_{t\in[0,T]}\|\rho(t)-\rho_L(t)\|_{L^2(\Omega)}\to 0.
\]
Hence \(\rho(0)=\rho_0\) in \(L^2(\Omega)\); in particular \(\rho\in C([0,T];L^2(\Omega))\).

\smallskip
\noindent\textit{Mass conservation.}
By Step~4 we have \(\int_\Omega\rho_L(t)\dx=\int_\Omega\rho_0\dx\) for every \(t\in[0,T]\),
and \(\rho_L(t)\to\rho(t)\) in \(L^2(\Omega)\) for every \(t\). Passing to the limit yields
\[
\int_\Omega \rho(t)\dx=\int_\Omega \rho_0\dx
\qquad\text{for all }t\in[0,T],
\]
which is the asserted mass conservation.

\smallskip
\noindent\textit{Weak mass balance.}
It remains to recast the mass balance in the form \eqref{eq:weak_mass}. Passing to the limit
\(L\to\infty\) in \eqref{eq:weak_mass_L}, using \(J_L\rightharpoonup J\) in \(L^{16/9}(\Omega_T)\) and
\(\partial_t\rho_L\rightharpoonup\partial_t\rho\) in \(L^{16/9}(0,T;(W^{1,16/7}(\Omega))')\), we obtain, for all
\(\varphi\in L^{16/7}(0,T;W^{1,16/7}(\Omega))\),
\[
\int_0^T \langle \partial_t\rho,\varphi\rangle\dt
+\int_0^T\!\!\int_\Omega J\cdot\nabla\varphi\dx\dt=0.
\]
Let \(\psi\in C_c^\infty([0,T)\times\overline\Omega)\). Then \(\psi\in L^{16/7}(0,T;W^{1,16/7}(\Omega))\),
and since \(\rho\in C([0,T];L^2(\Omega))\) with \(\rho(0)=\rho_0\), integration by parts in time yields
\[
\int_0^T \langle \partial_t\rho,\psi\rangle\dt
=
-\int_0^T\!\!\int_\Omega \rho\,\partial_t\psi\dx\dt
-\int_\Omega \rho_0\,\psi(\cdot,0)\dx.
\]
Combining the last two identities gives \eqref{eq:weak_mass}.
This completes the proof of Theorem~\ref{thm:main}.
\qed

\section{Conclusion and outlook}\label{sec:conclusion}

We have proved the global-in-time existence of nonnegative weak solutions for a one-sided degenerate Cahn--Hilliard equation with prescribed anisotropic traction. The analysis combines a weighted energy estimate with a mobility-matched entropy method. The entropy structure is crucial for recovering higher-order compactness and for preserving nonnegativity in the degenerate limit, while the weighted energy estimate identifies the diffusive flux in the form required by the weak formulation.

Several questions remain open. First, uniqueness is not addressed here and appears to be delicate for this class of degenerate fourth-order equations, especially because the anisotropic traction term is not generated by the variational derivative of the Cahn--Hilliard energy. Second, the long-time behavior and coarsening properties may depend sensitively on the geometry and regularity of the prescribed tensor \(D(\nu)\). Understanding whether anisotropic traction can select persistent length scales or alter classical Cahn--Hilliard coarsening rates, such as the Lifshitz--Slyozov \(t^{1/3}\) rate, would be an interesting direction for future work. Finally, the present analysis treats the orientation field as prescribed. Natural extensions include density-dependent traction tensors, time-dependent orientation fields coupled to additional transport or Liouville-type equations, and the full two-population model before reduction to the segregated one-field regime.

{\small
\bibliographystyle{ieeetr}
\bibliography{literature}}

\end{document}